\documentclass{amsproc}

\usepackage{stmaryrd}
\usepackage{graphicx}
\usepackage{amssymb}
\usepackage{amscd}
\usepackage{amsmath}
\usepackage{amsfonts}
\usepackage{amsbsy}
\usepackage{epsfig}
\usepackage{tikz}
\usepackage{cite}
\usepackage{indentfirst}
\usepackage{multicol}
\usepackage{wasysym}
\usepackage{bbm}
\usepackage{hyperref}
\usepackage{longtable,booktabs}
\usepackage{bm}

\newtheorem{thm}{Theorem}[section]%
\newtheorem{cor}[thm]{Corollary}%
\theoremstyle{definition}
\newtheorem{defi}[thm]{Definition}%
\newtheorem{pro}[thm]{Proposition}%
\theoremstyle{remark}
\newtheorem{rem}[thm]{Remark}%
\newtheorem*{thm A}{Theorem A}
\newtheorem*{thm B}{Theorem B}

\allowdisplaybreaks[4]

\numberwithin{equation}{section}

\begin{document}
\title[finite-dimensional Hopf algebras over $H_{b:1}^*$]{ On some classification
of finite-dimensional Hopf algebras over the Hopf algebra $H_{b:1}^*$ of Kashina}

\author{Yiwei Zheng$^*$}
\address{School of Mathematics, Hangzhou Normal University,
Hangzhou 311121, China}
\email{zhengyiwei12@foxmail.com}

\author{Yun Gao}
\address{Department of Mathematics and Statistics
, York University, Toronto M3J 1P3, Canada}
\email{ygao@yorku.ca}

\author{Naihong Hu}
\address{School of Mathematical Sciences, Shanghai Key Laboratory of PMMP, East China Normal University,
Shanghai 200241, China}
\email{nhhu@math.ecnu.edu.cn}
\thanks{$^*$ Corresponding author.}

\author{Yuxing Shi}
\address{College of Mathematics and Information Science, Jiangxi Normal University, Nanchang 330022, China}
\email{blueponder@foxmail.com}

%
%

\begin{abstract}
Let $H$ be the dual of $16$-dimensional nontrivial semisimple Hopf
algebra $H_{b:1}$ in the classification work of Kashina \cite{K00}.
We completely determine all finite-dimensional Nichols algebras
satisfying $\mathcal{B}(N)\cong \bigotimes_{i\in
I}\mathcal{B}(N_i)$, where $N=\bigoplus_{i\in I}N_i$, each $N_i$ is
a simple object in $_H^H\mathcal{YD}$.
Under this assumption, we classify all those Hopf algebras of finite-dimensional
growth from the semisimple Hopf algebra $H$ via the relevant Nichols algebras $\mathcal B(N)$.
\end{abstract}

\maketitle

\section{Introduction}

Let $\mathbbm{k}$ be an algebraically closed field of characteristic
zero. This work is committed to classifying finite-dimensional Hopf
algebras over $\mathbbm{k}$ with a nontrivial
semisimple Hopf algebra (here ``nontrivial'' means not those group
algebras or their duals) as the coradical.
Until now, there are few classification results on such Hopf
algebras, except \cite{AGM15}, \cite{S16}, \cite{Z}, \cite{Z2}, \cite{S20}.

In this paper, we will use the lifting  method observed by
Andruskiewitsch and Schneider in their construction for a class of
finite-dimensional pointed Hopf algebras arising from quantum linear
spaces \cite{AS}. Actually, the lifting   method works well for
those classes of finite-dimensional Hopf algebras with Radford
biproduct structure or their deformations by certain Hopf
$2$-cocycles. Let $A$ be a Hopf algebra. If the coradical $A_0$ is a
Hopf subalgebra, then the coradical filtration $\{A_n\}_{n\geq 0}$ is a
Hopf algebra filtration, and the associated graded coalgebra
$grA=\bigoplus_{n\geq 0}A_n/A_{n-1}$ is also a Hopf algebra, where
$A_{-1}=0$. Let $\pi:gr A\rightarrow A_0$ be the homogeneous
projection. By a theorem of Radford\ \cite{R85}, there exists a
unique connected graded braided Hopf algebra $R=\bigoplus_{n\geq
0}R(n)$ in the monoidal category $^{A_0}_{A_0}\mathcal{YD}$ such
that $grA\cong R\sharp A_0$. We call $R$ and $R(1)$ the diagram and
infinitesimal braiding of $A$, respectively. Moreover, $R$ is
strictly graded, that is, $R(0)=\mathbbm{k}, ~R(1)=\mathcal{P}(R)$
(see Definition 1.13 \cite{AS2}).

The lifting  method has been applied to classify some
finite-dimensional pointed Hopf algebras such as
\cite{ACG15a},\ \cite{ACG15b},\ \cite{AFG10},\ \cite{AFG11},\ \cite{AHS},\ \cite{AS3},\ \cite{FG},\ \cite{GGI},
etc., and copointed Hopf algebras \cite{AV},\ \cite{GIV},\ \cite{FGM}, etc.
Nevertheless, there are a few classification results on
finite-dimensional Hopf algebras whose coradical is neither a group
algebra nor the dual of a group algebra, for instance, \cite{AGM15},
\cite{CS}, \cite{GG16}, \cite{HX16}, \cite{HX18}, \cite{S16},
\cite{X19}, \cite{Z}, \cite{Z2}, etc. In particular, the 4th author
made an attempt in \cite{S16} to classify the objects of
finite-dimensional growth from a given nontrivial semisimple Hopf
algebra $A_0=H_8$ via some relevant Nichols algebras $\mathcal B(V)$
derived from its semisimple Yetter-Drinfeld modules $V\in
{^{A_0}_{A_0}\mathcal{YD}}$.

In \cite{AGM15}, authors constructed some examples of Hopf algebras
with the dual Chevalley property by determining all semisimple Hopf
algebras that are Morita-equivalent to finite group algebras. But in
our case, the semisimple Hopf algebra $H_{b:1}^*$ is a
$2$-pseudo-cocycle twist (rather than a $2$-cocycle twist) of some group algebra \cite{K00}, so it is no longer Morita-equivalent to a
group algebra as semisimple algebras. In some sense, the program is a
new attempt at aiming to find and classify the unknown
finite-dimensional Hopf algebras starting with a given {\it
nontrivial} semisimple algebra (rather than a (dual) group algebra).

Now we outline the strategy given by the lifting method. We first fix a finite-dimensional nontrivial semisimple Hopf algebra
$H$. Then we determine those $V$ in $^{H}_{H}\mathcal{YD}$ such that the Nichols algebra $\mathcal{B}(V)$ is finite dimensional and
present $\mathcal{B}(V)$ by generators and relations. Finally, we calculate all possible Hopf algebras $A$ such that $grA \cong \mathcal{B}(V)\sharp H$.
We call $A$ a lifting of $\mathcal{B}(V)$ over $H$.
%
%
%

The present paper is a sequel to \cite{S16}, \cite{Z}, \cite{Z2}. In
this paper, we let $H$ be the dual of $H_{b:1}$, which appeared in
\cite{K00}. For the definition of Hopf algebra $H$, see Definition
\ref{def3.1}. Since $^{H^\ast}_{H^\ast}\mathcal{YD}$ is braided
equivalent to $^H_H\mathcal{YD}$, Nichols algebras in
$^{H^\ast}_{H^\ast}\mathcal{YD}$ are equivalent to Nichols algebras
in $^H_H\mathcal{YD}$. Then we have the same Theorem A as in
\cite{Z} for such a $H$.

%
%
\begin{thm A}{\rm\cite [Theorem A]{Z}}
Let $N=\bigoplus_{i\in I}N_i$, where $N_i$ is a simple module in
$_H^H\mathcal{YD}$. Then the Nichols algebra $\mathcal{B}(N)$
satisfying $\mathcal{B}(N)\cong \bigotimes_{i\in I}\mathcal{B}(N_i)$
is finite-dimensional if and only if $N$ is isomorphic to one of the
following Yetter-Drinfeld modules


$(1)\ \Omega_1(n_1,n_2,n_3,n_4,n_5,n_6,n_7,n_8)\triangleq\bigoplus_{j=1}^8 V_j^{n_j}$ with $\sum_{j=1}^8n_j\geq 1$.

$(2)\ \Omega_2(n_1,n_2,n_3,n_4)\triangleq V_5^{\oplus n_1}\oplus V_6^{\oplus n_2}\oplus V_7^{\oplus n_3}\oplus V_8^{\oplus n_4}\oplus M_1$ with $\sum_{j=1}^4n_j\geq 0$.

$(3)\ \Omega_3(n_1,n_2,n_3,n_4)\triangleq V_1^{\oplus n_1}\oplus V_2^{\oplus n_2}\oplus V_3^{\oplus n_3}\oplus V_4^{\oplus n_4}\oplus M_2$ with $\sum_{j=1}^4n_j\geq 0$.

$(4)\ \Omega_4(n_1,n_2,n_3,n_4)\triangleq V_1^{\oplus n_1}\oplus V_2^{\oplus n_2}\oplus V_5^{\oplus n_3}\oplus V_6^{\oplus n_4}\oplus M_3$ with $\sum_{j=1}^4n_j\geq 0$.

$(5)\ \Omega_5(n_1,n_2,n_3,n_4)\triangleq V_1^{\oplus n_1}\oplus V_2^{\oplus n_2}\oplus V_7^{\oplus n_3}\oplus V_8^{\oplus n_4}\oplus M_4$ with $\sum_{j=1}^4n_j\geq 0$.

$(6)\ \Omega_6(n_1,n_2,n_3,n_4)\triangleq V_3^{\oplus n_1}\oplus V_4^{\oplus n_2}\oplus V_7^{\oplus n_3}\oplus V_8^{\oplus n_4}\oplus M_5$ with $\sum_{j=1}^4n_j\geq 0$.

$(7)\ \Omega_7(n_1,n_2,n_3,n_4)\triangleq V_3^{\oplus n_1}\oplus V_4^{\oplus n_2}\oplus V_5^{\oplus n_3}\oplus V_6^{\oplus n_4}\oplus M_6$ with $\sum_{j=1}^4n_j\geq 0$.

$(8)\ \Omega_{8}(n_1,n_2,n_3,n_4)\triangleq V_1^{\oplus n_1}\oplus V_2^{\oplus n_2}\oplus V_3^{\oplus n_3}\oplus V_4^{\oplus n_4}\oplus M_7$ with $\sum_{j=1}^4n_j\geq 0$.

$(9)\ \Omega_{9}(n_1,n_2,n_3,n_4)\triangleq V_5^{\oplus n_1}\oplus V_6^{\oplus n_2}\oplus V_7^{\oplus n_3}\oplus V_8^{\oplus n_4}\oplus M_{8}$ with $\sum_{j=1}^4n_j\geq 0$.

$(10)\ \Omega_{10}(n_1,n_2,n_3,n_4)\triangleq V_3^{\oplus n_1}\oplus V_4^{\oplus n_2}\oplus V_7^{\oplus n_3}\oplus V_8^{\oplus n_4}\oplus M_{9}$ with $\sum_{j=1}^4n_j\geq 0$.

$(11)\ \Omega_{11}(n_1,n_2,n_3,n_4)\triangleq V_3^{\oplus n_1}\oplus V_4^{\oplus n_2}\oplus V_5^{\oplus n_3}\oplus V_6^{\oplus n_4}\oplus M_{10}$ with $\sum_{j=1}^4n_j\geq 0$.

$(12)\ \Omega_{12}(n_1,n_2,n_3,n_4)\triangleq V_1^{\oplus n_1}\oplus V_2^{\oplus n_2}\oplus V_5^{\oplus n_3}\oplus V_6^{\oplus n_4}\oplus M_{11}$ with $\sum_{j=1}^4n_j\geq 0$.

$(13)\ \Omega_{13}(n_1,n_2,n_3,n_4)\triangleq V_1^{\oplus n_1}\oplus V_2^{\oplus n_2}\oplus V_7^{\oplus n_3}\oplus V_8^{\oplus n_4}\oplus M_{12}$ with $\sum_{j=1}^4n_j\geq 0$.

$(14)\ \Omega_{14}\triangleq M_1\oplus M_1$, $\Omega_{15}\triangleq M_1\oplus M_2$, $\Omega_{16}\triangleq M_1\oplus M_7$, $\Omega_{19}\triangleq M_3\oplus M_3$,

\hspace{2.2em}$\Omega_{20}\triangleq M_3\oplus M_5$, $\Omega_{21}\triangleq M_3\oplus M_9$, $\Omega_{22}\triangleq M_4\oplus M_4$, $\Omega_{23}\triangleq M_4\oplus M_6$,

\hspace{2.2em}$\Omega_{29}\triangleq M_7\oplus M_7$, $\Omega_{30}\triangleq M_7\oplus M_8$, $\Omega_{38}\triangleq M_{13}\oplus M_{13}$, $\Omega_{39}\triangleq M_{13}\oplus M_{14}$,

\hspace{2.2em}$\Omega_{41}\triangleq M_{15}\oplus M_{15}$, $\Omega_{42}\triangleq M_{15}\oplus M_{16}$, $\Omega_{44}\triangleq M_{17}\oplus M_{17}$, $\Omega_{45}\triangleq M_{17}\oplus M_{18}$.

$(15)\ \Omega_{17}\triangleq M_2\oplus M_2$, $\Omega_{18}\triangleq M_2\oplus M_8$, $\Omega_{24}\triangleq M_4\oplus M_{10}$, $\Omega_{25}\triangleq M_5\oplus M_5$,

\hspace{2em}$\Omega_{26}\triangleq M_5\oplus M_{11}$, $\Omega_{27}\triangleq M_6\oplus M_6$, $\Omega_{28}\triangleq M_6\oplus M_{12}$, $\Omega_{31}\triangleq M_8\oplus M_8$,

\hspace{2em}$\Omega_{32}\triangleq M_9\oplus M_9$, $\Omega_{33}\triangleq M_9\oplus M_{11}$, $\Omega_{34}\triangleq M_{10}\oplus M_{10}$, $\Omega_{35}\triangleq M_{10}\oplus M_{12}$,

\hspace{2em}$\Omega_{36}\triangleq M_{11}\oplus M_{11}$, $\Omega_{37}\triangleq M_{12}\oplus M_{12}$, $\Omega_{40}\triangleq M_{14}\oplus M_{14}$, $\Omega_{43}\triangleq M_{16}\oplus M_{16}$,

\hspace{2em}$\Omega_{46}\triangleq M_{18}\oplus M_{18}$, $\Omega_{47}\triangleq M_{19}\oplus M_{19}$,
$\Omega_{48}\triangleq M_{19}\oplus M_{20}$,
$\Omega_{49}\triangleq M_{20}\oplus M_{20}$.
\end{thm A}
%
%

However, the crucial point of present paper focuses on the different
lifting procedures, for $H$ is not isomorphic to its dual Hopf
algebra $H_{b:1}$ in \cite{Z}. Based on the lifting  method, we
classify the finite-dimensional Hopf algebras over $H$ such that
their infinitesimal braidings are those Yetter-Drinfeld modules
listed in Theorem A. Here are our main classification results,
surprisingly, which are different than those listed in Theorem B \cite{Z}.

\begin{thm B}
Let $A$ be a finite-dimensional Hopf algebra over $H$ such that its
infinitesimal braiding is isomorphic to one of the Yetter-Drinfeld modules listed in Theorem A, then $A$ is isomorphic either to

$(1)~\mathfrak{U}_1(n_1,n_2,\ldots,n_8;I_1)$, see Definition $\ref{def 1}$;

$(2)~\mathfrak{U}_2(n_1,n_2,n_3,n_4;I_2)$, see Definition $\ref{def 2}$;

$(3)~\mathfrak{U}_9(n_1,n_2,n_3,n_4;I_9)$, see Definition $\ref{def 9}$;

$(4)~\mathfrak{U}_{14}(I_{14})$, see Definition $\ref{def 14}$; $\mathfrak{U}_{15}(I_{15})$, see Definition $\ref{def 15}$;

\hspace{1.6em}$\mathfrak{U}_{16}(\lambda)$, see Definition $\ref{def 16}$;
$\mathfrak{U}_{20}(\lambda)$, see Definition $\ref{def 20}$;

\hspace{1.6em}$\mathfrak{U}_{23}(\lambda)$, see Definition $\ref{def 23}$;
$\mathfrak{U}_{29}(\lambda)$, see Definition $\ref{def 29}$;

\hspace{1.6em}$\mathfrak{U}_{38}(I_{38})$, see Definition $\ref{def 38}$;
$\mathfrak{U}_{39}(I_{39})$, see Definition $\ref{def 39}$;

\hspace{1.6em}$\mathfrak{U}_{41}(I_{41})$, see Definition $\ref{def 41}$;
$\mathfrak{U}_{42}(I_{42})$, see Definition $\ref{def 42}$;

\hspace{1.6em}$\mathfrak{U}_{44}(I_{44})$, see Definition $\ref{def 44}$;
$\mathfrak{U}_{45}(I_{45})$, see Definition $\ref{def 45}$.

$(5)$ $\mathcal{B}(V)\sharp H$ for $V\in\{\Omega_{4},\ \Omega_{5},\ \Omega_{19},
\ \Omega_{21},\ \Omega_{22},\ \Omega_{30}\}$, see Proposition $\ref{19}$.
\end{thm B}

Notice that except for case
$(5)$, the remainder  families of Hopf algebras contain non-trivial lifting relations.
It is interesting to note the following distinct points between Theorem B aforementioned and Theorem B in \cite{Z}.

\begin{rem}
$(1)$ $\dim \mathfrak{U}_1(n_1,n_2,\ldots,n_8;I_1)=2^{4+\sum\limits _{i=1}^8 n_i}.$

$(2)$ For $k\in\{2,9\}$, $\dim \mathfrak{U}_k(n_1,n_2,n_3,n_4;I_k)=2^{6+\sum\limits _{i=1}^4 n_i}$.

$(3)$ For $i\in\{14,15,38,39,41,42,44,45\}$,
$\dim \mathfrak{U}_{i}(I_{i})=256$.

$(4)$ For $i\in\{16,20,23,29\}$,
$\dim \mathfrak{U}_{i}(\lambda)=256$.

$(5)$ For $V=\Omega_{39}$, the lifting of $\mathcal{B}(V)$ over $H_{b,1}$ is trivial, the lifting of $\mathcal{B}(V)$ over $H_{b:1}^*$ is non-trivial.

$(6)$ For $V\in\{\Omega_{4},\ \Omega_{5},\ \Omega_{19},
\ \Omega_{21},\ \Omega_{22},\ \Omega_{30}\}$, the lifting of $\mathcal{B}(V)$ over $H_{b,1}$
is non-trivial, the lifting of $\mathcal{B}(V)$ over $H_{b:1}^*$ is trivial.

$(7)$ The Hopf algebras appeared in Theorem B are not isomorphic to those listed in
Theorem B \cite{Z}. The reason is that the both group subalgebras generated by the respective group-like
elements are not isomorphic to each other. These cause the quite different lifting procedures.
\end{rem}

The paper is organized as follows. In section 2, we recall some
basics and notations of Yetter-Drinfeld modules, Nichols algebras,
Radford biproduct. In section $3$, we describe the structure of $H$
and determine the simple objects in $^H_H\mathcal{YD}$ whose Nichols
algebras are finite dimensional. In section 4, based on the lifting
 method, we classify the finite-dimensional Hopf algebras over $H$
such that their infinitesimal braidings are those Yetter-Drinfeld
modules listed in Theorem A. Then we get Theorem B.

\section{Preliminaries}

\noindent$\mathbf{Conventions}$.
Throughout the paper, the ground
field $\mathbbm{k}$ is always supposed to be an algebraically closed field of characteristic zero and, denote by $\xi$ a primitive $4$-th root of
unity. For references on Hopf algebra theory, one can consult
\cite{M}, \cite{CK}, \cite{R}, \cite{SW}, etc.

If $H$ is a Hopf algebra over $\mathbbm{k}$, then $\triangle,~\varepsilon, \ S$ denote the comultiplication, the counit and the antipode, respectively. We use Sweedler's notation for the comultiplication and coaction, $e.g.$, $\triangle(h)=h_{(1)}\otimes h_{(2)}$ for $h\in H$. Denote by $G(H)$ the set of group-like elements of $H$. For $g,\ h\in G(H)$, the linear space of $(g,h)$-primitives is$:$
\begin{equation*}
\mathcal{P}_{g,h}(H)=\{x\in H\mid \triangle(x)=x\otimes g+h\otimes x\}.
\end{equation*}
In particular, the linear space $\mathcal{P}(H)=\mathcal{P}_{1,1}(H)$ is called the set of primitive elements of $H$.

For $k$, $n$ integers such that $0\leq k\leq n,$ we denote $\mathbb{Z}_n=\mathbb{Z}/n\mathbb{Z}$ and $\mathbb{I}_{k,n}=\{k,k+1,...,n\}$.
In particular, the operations $ij$ and $i\pm j$ are considered modulo $n+1$ for $i,\ j\in \mathbb{I}_{k,n}$ when not specified.

\subsection{Yetter-Drinfeld modules and Nichols algebras}
Let $H$ be a Hopf algebra  with bijective antipode. A left Yetter-Drinfeld module $V$ over $H$ is a left $H$-module $(V,\cdot)$ and a left $H$-comodule $(V,\delta)$ with $\delta(v)=v_{(-1)}\otimes v_{(0)}\in H\otimes V$ for all $v\in V$, satisfying
\begin{center}
$\delta(h\cdot v)=h_{(1)}v_{(-1)}S(h_{(3)})\otimes h_{(2)}\cdot v_{(0)},\quad \forall \;v\in V,\ h\in H.$
\end{center}
We denote by $^H_H\mathcal{YD}$ the category of finite-dimensional left Yetter-Drinfeld modules over $H$. It is a braided monoidal category: for $V,\ W \in {^H_H\mathcal{YD}}$, the braiding $c_{V,W}:V\otimes W\rightarrow W\otimes V$ is given by
\begin{equation}
c_{V,W}(v\otimes w)=v_{(-1)}\cdot w\otimes v_{(0)},\hspace{1em}\forall\ v\in V,\ w\in W.\label{braing}
\end{equation}
In particular, $(V,c_{V,V})$ is a braided vector space, that is, $c:= c_{V,V}$ is a linear isomorphism satisfying the braid equation
\begin{equation*}
(c\otimes {\rm id})({\rm id} \otimes c)(c\otimes {\rm id})=({\rm id} \otimes c)(c\otimes {\rm id})({\rm id} \otimes c).
\end{equation*}
Moreover, $^H_H\mathcal{YD}$ is rigid. Since $H$ is finite-dimensional, $^{H^\ast}_{H^\ast}\mathcal{YD}$ is braided equivalent to $^H_H\mathcal{YD}$
(see \cite{AG99}). Let $\{h_i\}_{i\in \mathbb{I}_{0,n}}$ and $\{h^i\}_{i\in \mathbb{I}_{0,n}}$ be the dual bases of $H$ and $H^\ast$, respectively.
If $V\in {^H_H\mathcal{YD}}$, then $V\in {^{H^\ast}_{H^\ast}\mathcal{YD}}$ with
\begin{equation}\label{dual module}
f\cdot v=f(S(v_{(-1)}))v_{(0)},\ \delta(v)=\sum\limits_i S^{-1}(h^i)\otimes h_i\cdot v,\ \ \forall\ v\in V,\ f\in H^\ast.
\end{equation}

\begin{defi}{\rm\cite [Definition 2.1]{AS2}}
~Let $H$ be a Hopf algebra and $V$ a Yetter-Drinfeld module over $H$. A braided $\mathbb{N}$-graded Hopf algebra $R=\bigoplus_{n\geq 0}R(n)$ in $^H_H\mathcal{YD}$ is called a Nichols algebra of $V$ if
\begin{center}
$\mathbbm{k}\simeq R(0),\ V\simeq R(1),\ R(1)=\mathcal{P}(R),\ R$ is generated as an algebra by $R(1)$.
\end{center}

\end{defi}

For any $V\in {^H_H\mathcal{YD}}$ there is a Nichols algebra $\mathcal{B}(V)$ associated to it. It is the quotient of the tensor algebra $T(V)$ by the largest homogeneous two-sided ideal $I$ satisfying:

\ $\bullet\ I$ is generated by homogeneous elements of degree $\geq 2$.

\ $\bullet\ \triangle(I)\subseteq I\otimes T(V)+T(V)\otimes I$, i.e., it is also a coideal.
\\ In such case, $\mathcal{B}(V)=T(V)/I$. See {\rm\cite [Section 2.1]{AS2}} for details.

\begin{rem}\label{rem}
It is well known that the Nichols algebra $\mathcal{B}(V)$ is completely determined, as an algebra and a
coalgebra, by the braided space $V$.
If $W\subseteq V$ is a subspace such that $c(W\otimes W)\subseteq W\otimes W$, then $\dim \mathcal{B}(V)=\infty$ if $\dim \mathcal{B}(W)=\infty$.
In particular, if $V$ contains a non-zero element $v$ such that $c(v\otimes v)=v\otimes v$, then $\dim \mathcal{B}(V)=\infty$.
\end{rem}

\subsection{Bosonization and Hopf algebras with a projection}
Let $R$ be a Hopf algebra in $^H_H\mathcal{YD}$ and denote the coproduct by $\triangle_R(r)=r^{(1)}\otimes r^{(2)}$ for $r\in R$. We define the Radford biproduct or bosonization $R \sharp H$ as follows: as a vector space, $R \sharp H=R\otimes H$, and the multiplication and comultiplication are given by the smash product and smash-coproduct, respectively:
\begin{center}
$(r\sharp g)(s\sharp h)=r(g_{(1)}\cdot s)\sharp g_{(2)}h, \quad \triangle(r\sharp g)=r^{(1)}\sharp (r^{(2)})_{(-1)}g_{(1)}\otimes (r^{(2)})_{(0)}\sharp g_{(2)}$.
\end{center}
Clearly, the map
$\iota: H\rightarrow R \sharp H, \ h\mapsto 1\sharp h, \ \forall \ h\in H$,
is injective and the map
\begin{equation*}
\pi: R \sharp H \rightarrow H, \quad r\sharp h\mapsto \varepsilon_R(r)h, \quad \forall \ r\in R, \ h\in H,
\end{equation*}
is surjective such that $\pi \circ \iota={\rm id}_H$. Moreover, it holds that
\begin{equation*}
R\cong R\sharp 1=(R \sharp H)^{coH}=\{x\in R \sharp H\mid ({\rm id}\otimes \pi)\triangle(x)=x\otimes1\}.
\end{equation*}

Conversely, if $A$ is a Hopf algebra with bijective antipode and $\pi:A\rightarrow H$ is a
Hopf algebra epimorphism admitting a Hopf algebra section $\iota: H\rightarrow A$
such that $\pi \circ \iota={\rm id}_H$, then $R=A^{co \pi}$ is a braided Hopf algebra in $^H_H\mathcal{YD}$ and $A\cong R \sharp H$ as Hopf algebras. See \cite{R85} and \cite{R} for more details.

\section{Nichols algebras of semisimple modules in $^H_H\mathcal{YD}$}
In this section, we recall the structure of $H=H_{b:1}^*$ given in \cite{Z}
and determine the simple objects in $^H_H\mathcal{YD}$ whose Nichols algebras are finite dimensional
by using the facts and results in \cite{Z}.
Let $N=\bigoplus_{i\in I}N_i$, where $N_i$ is a simple module in $_H^H\mathcal{YD}$. Then we try to determine all the finite-dimensional Nichols algebras satisfying $\mathcal{B}(N)\cong \bigotimes_{i\in I}\mathcal{B}(N_i)$.

\begin{defi}{\rm\cite [Remark 3.2]{Z}}\label{def3.1}
As an algebra, $H$ is generated by $a,\ b,\ c,\ d$, satisfying the relations
\begin{eqnarray}
a^2=b^2=c^2=1,\ ab=ba, \ ac=ca,\ bc=cb,\label{3.1}\\
d^2=a,\ \ da=ad,\ \ db=cd,\ \ dc=bd,
\end{eqnarray}
and its coalgebra is given by
\begin{flalign}
&\triangle(a)=a\otimes a,\ \triangle(b)=b\otimes b,\ \triangle(c)=c\otimes c,\ \varepsilon(a)=\varepsilon(b)=1,\\
&\triangle(d)=\frac{1}{2}[(1+bc)d\otimes d+(1-bc)d\otimes ad],\ \ \varepsilon(c)=\varepsilon(d)=1,\label{3.2}
\end{flalign}
and its antipode is given by
\begin{equation}
S(a)=a,\ S(b)=b,\ S(c)=c,\ S(d)=\frac{1}{2}[a(1+bc)+(1-bc)]d.\nonumber
\end{equation}
\end{defi}

\begin{rem}
$G(H)=\langle a\rangle \times \langle b\rangle\times \langle c\rangle$.
\end{rem}

\begin{pro}
The automorphism group of $H$ is given by
\begin{flalign*}
\langle\tau_2,\tau_5,\tau_9,\tau_{17}\mid &\tau_2^2=\tau_5^2=\tau_9^2=1,\tau_{17}^4=1,\tau_2\tau_5=\tau_5\tau_2,
\tau_2\tau_9=\tau_9\tau_2,\\
&\tau_2\tau_{17}=\tau_{17}\tau_2,\tau_5\tau_9=\tau_9\tau_5,
\tau_5\tau_{17}=\tau_{17}\tau_5=\tau_9\tau_{17}\tau_9\rangle\\
&=\langle\tau_2\rangle\times\langle\tau_5,\tau_9,\tau_{17}\rangle\cong \mathbb{Z}_2\times M_2(2,1,1),
\end{flalign*}
where $M_2(2,1,1)$ is the minimal nonabelian 2-group,  which is isomorphic
to the semidirect product $(\mathbb{Z}_4\times \mathbb{Z}_2)\rtimes \mathbb{Z}_2$.
Then all automorphisms of $H$ are given in Table $1$.

\renewcommand\arraystretch{1.5}
\begin{longtable}[c]{|c|c|c|c|c|c|c|c|c|c|}
  \hline
  \hspace{.5em}& \hspace{.5em}$a$\hspace{.5em} & \hspace{.5em}$b$\hspace{.5em} & \hspace{.5em}$c$\hspace{.5em} & \hspace{.5em}$d$\hspace{.5em} & \hspace{.5em} &\hspace{.5em}$a$\hspace{.5em} & \hspace{.5em}$b$\hspace{.5em} & \hspace{.5em}$c$\hspace{.5em} & \hspace{.5em}$d$ \\  %
  \hline
  $\tau_1$ & $a$ & $b$ & $c$ & $d$  &$\tau_{17}$ & $abc$ & $b$ & $c$ & $\frac{1}{2}[(1+\xi)+(1-\xi)bc]d$ \\
  \hline
  $\tau_2$ & $a$ & $b$ & $c$ & $da$ &$\tau_{18}$ & $abc$ & $b$ & $c$ & $\frac{1}{2}[(1-\xi)+(1+\xi)bc]d$\\
  \hline
  $\tau_3$ & $a$ & $b$ & $c$ & $dbc$ &$\tau_{19}$ & $abc$ & $b$ & $c$ & $\frac{1}{2}[(1+\xi)+(1-\xi)bc]ad$\\
  \hline
  $\tau_4$ & $a$ & $b$ & $c$ & $dabc$ &$\tau_{20}$ & $abc$ & $b$ & $c$ & $\frac{1}{2}[(1-\xi)+(1+\xi)bc]ad$\\
  \hline
  $\tau_5$ & $a$ & $c$ & $b$ & $d$ &$\tau_{21}$ & $abc$ & $c$ & $b$ & $\frac{1}{2}[(1+\xi)+(1-\xi)bc]d$\\
  \hline
  $\tau_6$ & $a$ & $c$ & $b$ & $da$ &$\tau_{22}$ & $abc$ & $c$ & $b$ & $\frac{1}{2}[(1-\xi)+(1+\xi)bc]d$ \\
  \hline
  $\tau_7$ & $a$ & $c$ & $b$ & $dbc$ &$\tau_{23}$ & $abc$ & $c$ & $b$ & $\frac{1}{2}[(1+\xi)+(1-\xi)bc]ad$\\
  \hline
  $\tau_8$ & $a$ & $c$ & $b$ & $dabc$ &$\tau_{24}$ & $abc$ & $c$ & $b$ & $\frac{1}{2}[(1-\xi)+(1+\xi)bc]ad$ \\
  \hline
  $\tau_9$ & $a$ & $ab$ & $ac$ & $d$ & $\tau_{25}$ & $abc$ & $ab$ & $ac$ & $\frac{1}{2}[(1+\xi)+(1-\xi)bc]d$\\
  \hline
  $\tau_{10}$ & $a$ & $ab$ & $ac$ & $da$ & $\tau_{26}$ & $abc$ & $ab$ & $ac$ & $\frac{1}{2}[(1-\xi)+(1+\xi)bc]d$\\
  \hline
  $\tau_{11}$ & $a$ & $ab$ & $ac$ & $dbc$& $\tau_{27}$ & $abc$ & $ab$ & $ac$ & $\frac{1}{2}[(1+\xi)+(1-\xi)bc]ad$\\
  \hline
  $\tau_{12}$ & $a$ & $ab$ & $ac$ & $dabc$
  & $\tau_{28}$ & $abc$ & $ab$ & $ac$ & $\frac{1}{2}[(1-\xi)+(1+\xi)bc]ad$\\
  \hline
  $\tau_{13}$ & $a$ & $ac$ & $ab$ & $d$
  & $\tau_{29}$ & $abc$ & $ac$ & $ab$ & $\frac{1}{2}[(1+\xi)+(1-\xi)bc]d$\\
  \hline
  $\tau_{14}$ & $a$ & $ac$ &  $ab$& $da$  &$\tau_{30}$ & $abc$ & $ac$ & $ab$ & $\frac{1}{2}[(1-\xi)+(1+\xi)bc]d$\\
  \hline
  $\tau_{15}$ & $a$ & $ac$ & $ab$ & $dbc$& $\tau_{31}$ & $abc$ & $ac$ & $ab$ & $\frac{1}{2}[(1+\xi)+(1-\xi)bc]ad$\\
  \hline
  $\tau_{16}$ & $a$ & $ac$ & $ab$ & $dabc$& $\tau_{32}$ & $abc$ & $ac$ & $ab$ & $\frac{1}{2}[(1-\xi)+(1+\xi)bc]ad$\\
  \hline
\end{longtable}
\begin{center}
Table $1.$ Automorphisms of $H$
\end{center}
\end{pro}
\begin{proof}
Let $f$ be an automorphism of $H$. Then $f^*$ is an automorphism of $H^*$.
By Proposition\ 3.3 in \cite{Z}, the claim follows.
\end{proof}


Now we describe the simple objects in $^H_H\mathcal{YD}$ whose Nichols algebras are finite dimensional.
\begin{pro}\label{simple}
If $V$ is a simple object in $^H_H\mathcal{YD}$, then $\mathcal{B}(V)$ is finite-dimensional if and only if $V$ is isomorphic either to
one of the following Yetter-Drinfeld modules

$(1) \ \mathbbm{k}_{\chi_{i,j,k,l}}=\mathbbm{k}\{v\}$, $(i,j,l)\in\{(1,1,0),(1,3,0),(0,1,1),(0,3,1)\}$, $0\leq k<2$ with the structure given by
\begin{equation*}
a\cdot v=(-1)^jv,\ b\cdot v=(-1)^lv,\ c\cdot v=(-1)^lv,\ d\cdot v=\xi^{-j} v,\ \delta(v)=a^ib^{j+k}c^k\otimes v.
\end{equation*}

$(2)$ $M_1=\mathbbm{k}\{v_1,v_2\}$ with the structure given by
$\delta(v_1)=b\otimes v_1$, $\delta(v_2)=c\otimes v_2$ and
\begin{equation*}
[a]=
\left(
  \begin{array}{cc}
    1 & 0 \\
    0 & 1 \\
  \end{array}
\right),
\ [b]=
\left(
  \begin{array}{cc}
    -1& 0 \\
    0 & -1 \\
  \end{array}
\right),
\ [c]=
\left(
  \begin{array}{cc}
    -1& 0 \\
    0 & -1 \\
  \end{array}
\right),
\ [d]=
\left(
  \begin{array}{cc}
    0& 1 \\
    1 & 0 \\
  \end{array}
\right).
\end{equation*}

$(3)$ $M_2=\mathbbm{k}\{v_1,v_2\}$ with the coaction given by
$\delta(v_1)=ab\otimes v_1$, $\delta(v_2)=ac\otimes v_2$ and the
same action as $M_1$.

$(4)$ $M_3=\mathbbm{k}\{v_1,v_2\}$ with the structure given by
$\delta(v_1)=bc\otimes v_1$, $\delta(v_2)=abc\otimes v_2$ and
\begin{equation*}
[a]=
\left(
  \begin{array}{cc}
    1 & 0 \\
    0 & 1 \\
  \end{array}
\right),
\ [b]=
\left(
  \begin{array}{cc}
    1& 0 \\
    0 & -1 \\
  \end{array}
\right),
\ [c]=
\left(
  \begin{array}{cc}
    -1& 0 \\
    0 & 1 \\
  \end{array}
\right),
\ [d]=
\left(
  \begin{array}{cc}
    0& 1 \\
    1 & 0 \\
  \end{array}
\right).
\end{equation*}

$(5)$ $M_4=\mathbbm{k}\{v_1,v_2\}$ with its coaction given by
$\delta(v_1)=c\otimes v_1$, $\delta(v_2)=ab\otimes v_2$ and the same
action as $M_3$.

$(6)$ $M_5=\mathbbm{k}\{v_1,v_2\}$ with the structure given by
$\delta(v_1)=bc\otimes v_1$, $\delta(v_2)=abc\otimes v_2$ and
\begin{equation*}
[a]=
\left(
  \begin{array}{cc}
    1 & 0 \\
    0 & 1 \\
  \end{array}
\right),
\ [b]=
\left(
  \begin{array}{cc}
    -1& 0 \\
    0 & 1 \\
  \end{array}
\right),
\ [c]=
\left(
  \begin{array}{cc}
    1& 0 \\
    0 & -1 \\
  \end{array}
\right),
\ [d]=
\left(
  \begin{array}{cc}
    0& 1 \\
    1 & 0 \\
  \end{array}
\right).
\end{equation*}

$(7)$ $M_6=\mathbbm{k}\{v_1,v_2\}$ with its coaction given by
$\delta(v_1)=b\otimes v_1$, $\delta(v_2)=ac\otimes v_2$ and the same
action as $M_5$.

$(8)$ $M_7=\mathbbm{k}\{v_1,v_2\}$ with the structure given by
$\delta(v_1)=a\otimes v_1$, $\delta(v_2)=abc\otimes v_2$ and
\begin{equation*}
[a]=
\left(
  \begin{array}{cc}
    -1 & 0 \\
    0 & -1 \\
  \end{array}
\right),
\ [b]=
\left(
  \begin{array}{cc}
    1& 0 \\
    0 & 1 \\
  \end{array}
\right),
\ [c]=
\left(
  \begin{array}{cc}
    1& 0 \\
    0 & 1 \\
  \end{array}
\right),
\ [d]=
\left(
  \begin{array}{cc}
    0& -1 \\
    1 & 0 \\
  \end{array}
\right).
\end{equation*}

$(9)$ $M_8=\mathbbm{k}\{v_1,v_2\}$ with the structure given by
$\delta(v_1)=a\otimes v_1$, $\delta(v_2)=abc\otimes v_2$ and
\begin{equation*}
\hspace{1em}[a]=
\left(
  \begin{array}{cc}
    -1 & 0 \\
    0 & -1 \\
  \end{array}
\right),
\ [b]=
\left(
  \begin{array}{cc}
    -1& 0 \\
    0 & -1 \\
  \end{array}
\right),
\ [c]=
\left(
  \begin{array}{cc}
    -1& 0 \\
    0 & -1 \\
  \end{array}
\right),
\ [d]=
\left(
  \begin{array}{cc}
    0& -1 \\
    1 & 0 \\
  \end{array}
\right).
\end{equation*}

$(10)$ $M_9=\mathbbm{k}\{v_1,v_2\}$ with the structure given by
$\delta(v_1)=c\otimes v_1$, $\delta(v_2)=ac\otimes v_2$ and
\begin{equation*}
\ \ [a]=
\left(
  \begin{array}{cc}
    -1 & 0 \\
    0 & -1 \\
  \end{array}
\right),
\ [b]=
\left(
  \begin{array}{cc}
    1& 0 \\
    0 & -1 \\
  \end{array}
\right),
\ [c]=
\left(
  \begin{array}{cc}
    -1& 0 \\
    0 & 1 \\
  \end{array}
\right),
\ [d]=
\left(
  \begin{array}{cc}
    0& 1 \\
    -1 & 0 \\
  \end{array}
\right).
\end{equation*}

$(11)$ $M_{10}=\mathbbm{k}\{v_1,v_2\}$ with its coaction given by
$\delta(v_1)=bc\otimes v_1$, $\delta(v_2)=a\otimes v_2$ and the same
action as $M_9$.

$(12)$ $M_{11}=\mathbbm{k}\{v_1,v_2\}$ with the structure given by
$\delta(v_1)=b\otimes v_1$, $\delta(v_2)=ab\otimes v_2$ and
\begin{equation*}
[a]=
\left(
  \begin{array}{cc}
    -1 & 0 \\
    0 & -1 \\
  \end{array}
\right),
\ [b]=
\left(
  \begin{array}{cc}
    -1& 0 \\
    0 & 1 \\
  \end{array}
\right),
\ [c]=
\left(
  \begin{array}{cc}
    1& 0 \\
    0 & -1 \\
  \end{array}
\right),
\ [d]=
\left(
  \begin{array}{cc}
    0& 1 \\
    -1 & 0 \\
  \end{array}
\right).
\end{equation*}

$(13)$ $M_{12}=\mathbbm{k}\{v_1,v_2\}$ with its coaction given by
$\delta(v_1)=bc\otimes v_1$, $\delta(v_2)=a\otimes v_2$ and the same
action as $M_{11}$.

$(14)$ $M_{13}=\mathbbm{k}\{v_1,v_2\}$ with its action given by
\begin{equation*}
[a]=
\left(
  \begin{array}{cc}
    1 & 0 \\
    0 & 1 \\
  \end{array}
\right),
\ [b]=
\left(
  \begin{array}{cc}
    -1& 0 \\
    0 & 1 \\
  \end{array}
\right),
\ [c]=
\left(
  \begin{array}{cc}
    -1& 0 \\
    0 & 1 \\
  \end{array}
\right),
\ [d]=
\left(
  \begin{array}{cc}
    1& 0 \\
    0 & -1 \\
  \end{array}
\right),
\end{equation*}

and its coaction given by
\begin{flalign*}
&\delta(v_1)=\frac{1}{2}[a(b+c)d\otimes v_1+a(b-c)d\otimes v_2],\\
&\delta(v_2)=\frac{1}{2}[(b+c)d\otimes v_2+(b-c)d\otimes v_1].
\end{flalign*}

$(15)$ $M_{14}=\mathbbm{k}\{v_1,v_2\}$ with the same action as
$M_{13}$ and its coaction given by
\begin{flalign*}
&\delta(v_1)=\frac{1}{2}[(b+c)d\otimes v_1+(b-c)d\otimes v_2],\\
&\delta(v_2)=\frac{1}{2}[a(b+c)d\otimes v_2+a(b-c)d\otimes v_1].
\end{flalign*}

$(16)$ $M_{15}=\mathbbm{k}\{v_1,v_2\}$ with its action given by
\begin{equation*}
[a]=
\left(
  \begin{array}{cc}
    1 & 0 \\
    0 & 1 \\
  \end{array}
\right),
\ [b]=
\left(
  \begin{array}{cc}
    1& 0 \\
    0 & -1 \\
  \end{array}
\right),
\ [c]=
\left(
  \begin{array}{cc}
    1& 0 \\
    0 & -1 \\
  \end{array}
\right),
\ [d]=
\left(
  \begin{array}{cc}
    -1& 0 \\
    0 & -1 \\
  \end{array}
\right),
\end{equation*}

and its coaction given by
\begin{flalign*}
&\delta(v_1)=\frac{1}{2}[a(1+bc)d\otimes v_1+a(1-bc)d\otimes v_2],\\ &\delta(v_2)=\frac{1}{2}[(1+bc)d\otimes v_2+(1-bc)d\otimes v_1].
\end{flalign*}

$(17)$ $M_{16}=\mathbbm{k}\{v_1,v_2\}$ with the same action as
$M_{15}$ and its coaction given by
\begin{flalign*}
&\delta(v_1)=\frac{1}{2}[(1+bc)d\otimes v_1+(1-bc)d\otimes v_2],\\
&\delta(v_2)=\frac{1}{2}[a(1+bc)d\otimes v_2+a(1-bc)d\otimes v_1].
\end{flalign*}

$(18)$ $M_{17}=\mathbbm{k}\{v_1,v_2\}$ with the same coaction as $M_{15}$ and its action given by
\begin{equation*}
[a]=
\left(
  \begin{array}{cc}
    1 & 0 \\
    0 & 1 \\
  \end{array}
\right),
\ [b]=
\left(
  \begin{array}{cc}
    1& 0 \\
    0 & -1 \\
  \end{array}
\right),
\ [c]=
\left(
  \begin{array}{cc}
    -1& 0 \\
    0 & 1 \\
  \end{array}
\right),
\ [d]=
\left(
  \begin{array}{cc}
    0& -1 \\
    -1 & 0 \\
  \end{array}
\right).
\end{equation*}

$(19)$ $M_{18}=\mathbbm{k}\{v_1,v_2\}$ with the same action as
$M_{17}$ and the same coaction as $M_{16}$.

$(20)$ $M_{19}=\mathbbm{k}\{v_1,v_2\}$ with the same coaction as $M_{15}$ and its action given by
\begin{equation*}
[a]=
\left(
  \begin{array}{cc}
    -1 & 0 \\
    0 & -1 \\
  \end{array}
\right),
\ [b]=
\left(
  \begin{array}{cc}
    1& 0 \\
    0 & -1 \\
  \end{array}
\right),
\ [c]=
\left(
  \begin{array}{cc}
    -1& 0 \\
    0 & 1 \\
  \end{array}
\right),
\ [d]=
\left(
  \begin{array}{cc}
    0& 1 \\
    -1 & 0 \\
  \end{array}
\right).
\end{equation*}

$(21)$ $M_{20}=\mathbbm{k}\{v_1,v_2\}$ with the same coaction as $M_{15}$ and its action given by
\begin{equation*}
[a]=
\left(
  \begin{array}{cc}
    -1 & 0 \\
    0 & -1 \\
  \end{array}
\right),
\ [b]=
\left(
  \begin{array}{cc}
    -1& 0 \\
    0 & 1 \\
  \end{array}
\right),
\ [c]=
\left(
  \begin{array}{cc}
    1& 0 \\
    0 & -1 \\
  \end{array}
\right),
\ [d]=
\left(
  \begin{array}{cc}
    0& 1 \\
    -1 & 0 \\
  \end{array}
\right).
\end{equation*}
\end{pro}
\begin{proof}
Since $^{H^\ast}_{H^\ast}\mathcal{YD}$ is braided-equivalent to $^H_H\mathcal{YD}$ \cite{AG99}, by equation $(\ref{dual module})$ and Propositions $4.1$, $4.3$, $4.4$, $4.5$, Lemmas $5.3$, $5.4$, $5.5$ in \cite{Z},
we can get the claim.
\end{proof}

For simplifying our statements, we let
\begin{flalign*}
&V_1=\mathbbm{k}_{\chi_{1,1,0,0}},\hspace{1em} V_2=\mathbbm{k}_{\chi_{1,3,0,0}},
\hspace{1em}V_3=\mathbbm{k}_{\chi_{1,1,1,0}}, \hspace{1em}V_4=\mathbbm{k}_{\chi_{1,3,1,0}},\\
&V_5=\mathbbm{k}_{\chi_{0,1,0,1}}, \hspace{1em}V_6=\mathbbm{k}_{\chi_{0,3,0,1}},
\hspace{1em}V_7=\mathbbm{k}_{\chi_{0,1,1,1}}, \hspace{1em}V_8=\mathbbm{k}_{\chi_{0,3,1,1}}.
\end{flalign*}

%
%
%
%
Since $^{H^\ast}_{H^\ast}\mathcal{YD}$ is braided-equivalent to $^H_H\mathcal{YD}$ \cite{AG99},
we can get the following proposition.
\begin{pro}{\rm\cite [Proposition 5.6]{Z}}\label{sum}
$(1)$ Let $V,\ W$ be simple objects in $_H^H\mathcal{YD}$. As finite-dimensional Nichols algebras, we have $\mathcal{B}(V\oplus W)\cong \mathcal{B}(V)\otimes\mathcal{B}(W)$ in the following cases

\hspace{1.5em}$(a)$ $V, ~W\in \{V_1,V_2,\ldots, V_8\};$

\hspace{1.5em}$(b)$ $V\in \{V_5,V_6,V_7,V_8\}$, $W\in\{M_1, M_{8}\};$

\hspace{1.5em}$(c)$ $V\in \{V_1,V_2,V_3,V_4\}$, $W\in\{M_2, M_{7}\};$

\hspace{1.5em}$(d)$ $V\in \{V_1,V_2,V_5,V_6\}$, $W\in\{M_3, M_{11}\};$

\hspace{1.5em}$(e)$ $V\in \{V_1,V_2,V_7,V_8\}$, $W\in\{M_4, M_{12}\};$

\hspace{1.5em}$(f)$ $V\in \{V_3,V_4,V_7,V_8\}$, $W\in\{M_5, M_{9}\};$

\hspace{1.5em}$(g)$ $V\in \{V_3,V_4,V_5,V_6\}$, $W\in\{M_6, M_{10}\};$

\hspace{1.5em}$(h)$ $V=M_1$, $W\in\{M_2,M_7\}$; $V=M_3$ or $M_{11}$, $W\in\{M_5,M_9\};$

\hspace{1.5em}$(i)$ $V=M_4$ or $M_{12}$, $W\in\{M_6,M_{10}\};$

\hspace{1.5em}$(j)$ $(V,W)\in\{(M_2,M_8),(M_7,M_{8}),(M_{13},M_{14}),(M_{15},M_{16}),(M_{17},M_{18})$,

\hspace{27em} $(M_{19},M_{20})\};$

\hspace{1.5em}$(k)$ $V=W=M_i$ for $i\in\mathbb{I}_{1,20}$.

$(2)$ For $M\in \{M_{13},\ldots, M_{20}\},~ V\in\{V_1,\ldots,V_8\}$, $\dim \mathcal{B}(V\oplus M)=\infty$.

\end{pro}
%
$\mathbf{Proof\ of\ Theorem\ A}$. The claim follows by Propositions \ref{simple}, ~\ref{sum}.

\section{Hopf algebras over $H$}
In this section, based on the lifting method, we will determine all
finite-dimensio\-nal Hopf algebras over $H$ such that their
infinitesimal braidings are those Yetter-Drinfeld modules listed in
Theorem A.  Since $^{H^\ast}_{H^\ast}\mathcal{YD}$ is braided-equivalent to $^H_H\mathcal{YD}$, by {\rm\cite [Theorem 6.1]{Z}}, we
have that the diagrams of these Hopf algebras are Nichols algebras,
and by {\rm\cite [Lemma 6.1]{AS0}}, we get the following Corollary.

\begin{cor}{\rm\cite [Corollary 6.3]{Z}}
$(1)$ For
\begin{equation*}
(i,j)\in\{(2,3), (4,6), (5,7), (8,9), (10,12), (11,13), (4,11), (5,10)\},
\end{equation*}

$\mathcal{B}(\Omega_i(n_1,n_2,n_3,n_4))\sharp H\simeq \mathcal{B}(\Omega_j(n_1,n_2,n_3,n_4))\sharp H$.

$(2)$ $\mathcal{B}(\Omega_i)\sharp H\simeq \mathcal{B}(\Omega_j)\sharp H$, for
\begin{flalign*}
(i,j)\in\{&(14,17), (16,18), (19,25), (21,26), (22,27), (24,28), (29,31),\\ &(32,36),(34,37),(38,40), (41,43), (44,46), (47,49), (21,24),\\
&(22,32), (23,33), (19,34), (20,35), (44,47), (45,48)\}.
\end{flalign*}
\end{cor}

Let $i,\ j,\ k,\ \ell,\ m,\ q,\ s,\ r\in \mathbb{N}^*$, denote
\begin{flalign*}
&\mathbbm{k}A_i\simeq V_1,\hspace{1em} \mathbbm{k}B_j\simeq V_2, \hspace{1em}\mathbbm{k}C_k\simeq V_3, \hspace{1em}\mathbbm{k}D_\ell\simeq V_4,\\
&\mathbbm{k}E_m\simeq V_5,\hspace{1em} \mathbbm{k}F_q\simeq V_6, \hspace{1em}\mathbbm{k}G_s\simeq V_7, \hspace{1em}\mathbbm{k}H_r\simeq V_8.
\end{flalign*}

\begin{defi}\label{def 1}
For $n_1,n_2,\ldots,n_8\in \mathbb{N}^*$ with $\sum_{i=1}^8 n_i\geq 1$ and
\begin{flalign*}
I_1=\{\bm{\lambda}=(\lambda_{i,\ell})_{n_1\times n_4},\ \bm{\mu}=(\mu_{j,k})_{n_2\times n_3}, \ \bm{\nu}=(\nu_{m,r})_{n_5\times n_8}, \ \bm{\gamma}=(\gamma_{q,s})_{n_6\times n_7} \}
\end{flalign*}
with entries in $\mathbbm{k}$, let us denote by $\mathfrak{U}_1(n_1,n_2,\ldots,n_8;I_1)$ the algebra that is generated by $a,~b,~c,~d,$
$\{A_i\}_{i=1,...,n_1}$, $\{B_j\}_{j=1,...,n_2}$, $\{C_k\}_{k=1,...,n_3}$, $\{D_\ell\}_{\ell=1,...,n_4}$,
$\{E_m\}_{m=1,...,n_5}$, $\{F_q\}_{q=1,...,n_6}$, $\{G_s\}_{s=1,...,n_7}$, $\{H_r\}_{r=1,...,n_8}$ satisfying relations $(\ref{3.1})-(\ref{3.2})$ and
\begin{flalign}
& & aA_i=-A_ia, \ bA_i=A_ib, \ cA_i=A_ic, \ dA_i=-\xi A_id, \label{A}\\
& & aB_j=-B_ja, \ bB_j=B_jb, \ cB_j=B_jc, \ dB_j=\xi B_jd,\label{B}\\
& & aC_k=-C_ka, \ bC_k=C_kb, \ cC_k=C_kc, \ dC_k=-\xi C_kd,\label{C}\\
& & aD_\ell=-D_\ell a, \ bD_\ell=D_\ell b, \ cD_\ell=D_\ell c, \ dD_\ell=\xi D_\ell d,&\label{D}\\
& & aE_m=-E_ma, \ bE_m=-E_mb, \ cE_m=-E_mc, \ dE_m=-\xi E_md,\label{E}\\
& & aF_q=-F_qa, \ bF_q=-F_qb, \ cF_q=-F_qc, \ dF_q=\xi F_qd,\label{F}\\
& & aG_s=-G_sa, \ bG_s=-G_sb, \ cG_s=-G_sc, \ dG_s=-\xi G_sd,\label{G}\\
& & aH_r=-H_ra, \ bH_r=-H_rb, \ cH_r=-H_rc, \ dH_r=\xi H_rd,\label{H}\\
& & A_{i_1}A_{i_2}+A_{i_2}A_{i_1}=0,  \ i_1,i_2\in \{1,\ldots,n_1\},&\label{Ai}\\
& & B_{j_1}B_{j_2}+B_{j_2}B_{j_1}=0,  \ j_1,j_2\in \{1,\ldots,n_2\},&\label{Bj}\\
& & C_{k_1}C_{k_2}+C_{k_2}C_{k_1}=0,  \ k_1,k_2\in \{1,\ldots,n_3\},&\label{Ck}\\
& & D_{\ell_1}D_{\ell_2}+D_{\ell_2}D_{\ell_1}=0,  \ \ell_1,\ell_2\in \{1,\ldots,n_4\},&\label{Dl}\\
& & E_{m_1}E_{m_2}+E_{m_2}E_{m_1}=0,  \ m_1,m_2\in \{1,\ldots,n_5\},&\label{Em}\\
& & F_{q_1}F_{q_2}+F_{q_2}F_{q_1}=0,  \ q_1,q_2\in \{1,\ldots,n_6\},&\label{Fq}\\
& & G_{s_1}G_{s_2}+G_{s_2}G_{s_1}=0,  \ s_1,s_2\in \{1,\ldots,n_7\},&\label{Gs}\\
& & H_{r_1}H_{r_2}+H_{r_2}H_{r_1}=0,  \ r_1,r_2\in \{1,\ldots,n_8\},&\label{Hr}\\
& & A_iD_\ell+D_\ell A_i=\lambda_{i,\ell}(1-bc),\ B_jC_k+C_kB_j=\mu_{j,k}(1-bc),\label{AD}\\
& & E_mH_r+H_rE_m=\nu_{m,r}(1-bc),\ F_qG_s+G_sF_q=\gamma_{q,s}(1-bc),\label{EH}\\
& & A_iB_j+B_jA_i=0,\ A_iC_k+C_kA_i=0,\ A_iE_m-E_mA_i=0,\\
& & A_iF_q-F_qA_i=0,\ A_iG_s-G_sA_i=0,\ A_iH_r-H_rA_i=0,\\
& & B_jD_\ell+D_\ell B_j=0,\ B_jE_m-E_mB_j=0,\ B_jF_q-F_qB_j=0,\\
& & B_jG_s-G_sB_j=0,\ B_jH_r-H_rB_j=0,\ C_kD_\ell+D_\ell C_k=0,\\
& & C_kE_m-E_mC_k=0,\ C_kF_q-F_qC_k=0,\ C_kG_s-G_sC_m=0,\\
& & C_kH_r-H_rC_k=0,\ D_\ell E_m-E_mD_\ell=0,\ D_\ell F_q-F_qD_\ell=0,\\
& & D_\ell G_s-G_sD_\ell=0, \ D_\ell H_r-H_rD_\ell=0,\ E_mF_q+F_qE_m=0,\\
& & E_mG_s+G_sE_m=0,\  F_qH_r+H_rF_q=0, \ G_sH_r+H_rG_s=0.\label{e1}
\end{flalign}
It is a Hopf algebra with its coalgebra structure determined by
\begin{align*}
&\triangle(A_i)=A_i\otimes 1+ab\otimes A_i,\quad \triangle(B_j)=B_j\otimes 1+ab\otimes B_j,\\
&\triangle(C_k)=C_k\otimes 1+ac\otimes C_k,\quad \triangle(D_\ell)=D_\ell\otimes 1+ac\otimes D_\ell,\\
&\triangle(E_m)=E_m\otimes 1+b\otimes E_m,\quad \triangle(F_q)=F_q\otimes 1+b\otimes F_q,\\
&\triangle(G_s)=G_s\otimes 1+c\otimes G_s,\quad \triangle(H_r)=H_r\otimes 1+c\otimes H_r.
\end{align*}
\end{defi}

\begin{defi}\label{def 2}
For $n_1,n_2,n_3,n_4\in \mathbb{N}^*$ with $\sum_{i=1}^4 n_i\geq 0$,
\begin{flalign*}
I_2=\{\bm{\nu}=(\nu_{m,r})_{n_1\times n_4}, \ \bm{\gamma}=(\gamma_{q,s})_{n_2\times n_3},\ \lambda\},
\end{flalign*}
let us denote by $\mathfrak{U}_2(n_1,n_2,n_3,n_4;I_2)$ the algebra that is generated by $a,\ b,\ c,\ d$, $p_1,\ p_2$, $\{E_m\}_{m=1,...,n_1},$ $\{F_q\}_{q=1,...,n_2},$
$\{G_s\}_{s=1,\ldots,n_3}, \{H_r\}_{r=1,\ldots,n_4}$, satisfying the relations $(\ref{3.1})-(\ref{3.2})$, $(\ref{E})-(\ref{H})$, $(\ref{Em})-(\ref{Hr})$,
$(\ref{EH})$ and
\begin{flalign}
& & ap_1=p_1a,\hspace{1em}bp_1=-p_1b,\hspace{1em}cp_1=-p_1c,\hspace{1em}dp_1=p_2d,\label{e2.4}\\
& & ap_2=p_2a,\hspace{1em}bp_2=-p_2b,\hspace{1em}cp_2=-p_2c,\hspace{1em}dp_2=p_1d,\label{e2.5}\\
& & p_1^2=0,\ \ p_2^2=0,\ \ p_1p_2+p_2p_1=\lambda(1-bc),\label{e2.1}\\
& & p_1E_m+E_mp_1=0,\ p_1F_q+F_qp_1=0,\ p_1G_s+G_sp_1=0,\ p_1H_r+H_rp_1=0,\label{e2.2}\\
& & p_2E_m+E_mp_2=0,\ p_2F_q+F_qp_2=0,\ p_2G_s+G_sp_2=0,\ p_2H_r+H_rp_2=0.\label{e2.3}
\end{flalign}
It is a Hopf algebra with its coalgebra structure determined by
\begin{align*}
&\triangle(p_1)=p_1\otimes 1+b\otimes p_1,
\quad \quad \quad\triangle(p_2)=p_2\otimes 1+c\otimes p_2,\\
&\triangle(E_m)=E_m\otimes 1+b\otimes E_m,\hspace{1.8em}\triangle(F_q)=F_q\otimes 1+b\otimes F_q,\\
&\triangle(G_s)=G_s\otimes 1+c\otimes G_s,\hspace{2.5em}\triangle(H_r)=H_r\otimes 1+c\otimes H_r.
\end{align*}
\end{defi}

\begin{defi}\label{def 9}
For $n_1,\ n_2,\ n_3,\ n_4\in \mathbb{N}^*$ with $\sum_{i=1}^4 n_i\geq 0$ and
\begin{equation*}
I_9=\{ \bm{\nu} =(\nu_{m,r})_{n_1\times n_4}, \bm{\gamma} =(\gamma_{q,s})_{n_2\times n_3},\bm{\lambda},\bm{\mu},\bm{\alpha},\bm{\beta}\},
\end{equation*}
where
$\bm{\lambda} =(\lambda_{m})_m$, $\bm{\mu} =(\mu_q)_q$, $\bm{\alpha} =(\alpha_{s})_s$ and $\bm{\beta} =(\beta_{r})_r$.
Let us denote by $\mathfrak{U}_9(n_1,n_2,n_3,n_4;I_9)$ the algebra that is generated by $a,b,c,d,p_1,\ p_2,\{E_m\}_{m=1,...,n_1},$ $\{F_q\}_{q=1,...,n_2},$
$\{G_s\}_{s=1,\ldots,n_3}, \{H_r\}_{r=1,\ldots,n_4}$, satisfying the relations $(\ref{3.1})-(\ref{3.2})$, $(\ref{E})-(\ref{H})$, $(\ref{Em})-(\ref{Hr})$, $(\ref{EH})$
and
\begin{flalign}
& & ap_1=-p_1a,\hspace{1em}bp_1=-p_1b,\hspace{1em}cp_1=-p_1c,\hspace{1em}dp_1=p_2d,\nonumber\\
& & ap_2=-p_2a,\hspace{1em}bp_2=-p_2b,\hspace{1em}cp_2=-p_2c,\hspace{1em}dp_2=-p_1d,\nonumber\\
& & p_1^2=0,\ \ p_2^2=0,\ \ p_1p_2+p_2p_1=0,\nonumber\\
& & p_1E_m+E_mp_1=\lambda_m(1-ab),\ p_2E_m+E_mp_2=\xi\lambda_m(1-ac),\label{e9.1}\\
& & p_1F_q+F_qp_1=\mu_q(1-ab),\ p_2F_q+F_qp_2=-\xi\mu_q(1-ac),\label{e9.2}\\
& & p_1G_s+G_sp_1=\alpha_s(1-ac),\ p_2G_s+G_sp_2=\xi\alpha_s(1-ab),\label{e9.3}\\
& & p_1H_r+H_rp_1=\beta_r(1-ac),\ p_2H_r+H_rp_2=-\xi\beta_r(1-ab).\label{e9.4}
\end{flalign}
It is a Hopf algebra with the same coalgebra structure as $\mathfrak{U}_2(n_1,n_2,n_3,n_4;I_2)$ except for
\begin{align*}
&\triangle(p_1)=p_1\otimes 1+a\otimes p_1,
\hspace{3em}\triangle(p_2)=p_2\otimes 1+abc\otimes p_2.
\end{align*}
\end{defi}

Now we determine all the liftings under the  given assumption and discuss the isomorphism classes of the defining Hopf algebras.

\begin{pro}\label{1}
$(1)$ Let $A$ be a finite-dimensional Hopf algebra over $H$ such that its infinitesimal braiding is isomorphic to $\Omega_1(n_1,n_2,\ldots,n_8)$ or
$\Omega_i(n_1,n_2,n_3,n_4)$ for $i\in\{2,9\}$, then $A\cong \mathfrak{U}_1(n_1,n_2,\ldots,n_8;I_1)$ or $\mathfrak{U}_i(n_1,n_2,n_3,n_4;I_i)$.

$(2)\ \mathfrak{U}_1(n_1,\ldots,n_8;I_1)\cong
\mathfrak{U}_1(n_1,\ldots,n_8;I_1^{'})$ if and only if  there exist
invertible matrices
\begin{align*}
&\bm{a}=(a_{i^{'}i})_{n_1\times n_1},\ &\bm{b}&=(b_{j^{'}j})_{n_2\times n_2},
\ &\bm{c}&=(c_{k^{'}k})_{n_3\times n_3}, \ &\bm{d}&=(d_{\ell^{'}\ell})_{n_4\times n_4},\\
&\bm{e}=(e_{m^{'}m})_{n_5\times n_5}, \ &\bm{f}&=(f_{q^{'}q})_{n_6\times n_6},
\ &\bm{g}&=(g_{s^{'}s})_{n_7\times n_7},\ &\bm{h}&=(h_{r^{'}r})_{n_8\times n_8},
\end{align*}
such that
\begin{align}
&\bm{a}^T\bm{\lambda'}\bm{d}=\bm{\lambda}, \ &\bm{b}^T\bm{\mu'}\bm{c}&=\bm{\mu},\label{Aut1.1.1}\\
&\bm{e}^T\bm{\nu'}\bm{h}=\bm{\nu}, \ &\bm{f}^T\bm{\gamma'}\bm{g}&=\bm{\gamma},\label{Aut1.1.2}
\end{align}
or $n_1=n_2,\ n_3=n_4,\ n_5=n_6,\ n_7=n_8$
such that
\begin{align}
&\bm{a}^T\bm{\mu'}\bm{d}=\bm{\lambda}, \quad &\bm{b}^T\bm{\lambda'}\bm{c}&=\bm{\mu},\label{Aut1.2.1}\\
&\bm{e}^T\bm{\gamma'}\bm{h}=\bm{\nu}, \quad &\bm{f}^T\bm{\nu'}\bm{g}&=\bm{\gamma},\label{Aut1.2.2}
\end{align}
or $n_1=n_3,\ n_2=n_4,\ n_5=n_7,\ n_6=n_8$
such that
\begin{align}
&\bm{a}^T(\bm{\mu'})^T\bm{d}=\bm{\lambda}, \quad &\bm{b}^T(\bm{\lambda'})^T\bm{c}&=\bm{\mu},\label{Aut1.3.1}\\
&\bm{e}^T(\bm{\gamma'})^T\bm{h}=\bm{\nu}, \quad &\bm{f}^T(\bm{\nu'})^T\bm{g}&=\bm{\gamma},\label{Aut1.3.2}
\end{align}
or $n_1=n_4,\ n_2=n_3,\ n_5=n_8,\ n_6=n_7$
such that
\begin{align}
&\bm{a}^T(\bm{\lambda'})^T\bm{d}=\bm{\lambda}, \quad &\bm{b}^T(\bm{\mu'})^T\bm{c}&=\bm{\mu},\label{Aut1.4.1}\\
&\bm{e}^T(\bm{\nu'})^T\bm{h}=\bm{\nu}, \quad &\bm{f}^T(\bm{\gamma'})^T\bm{g}&=\bm{\gamma},\label{Aut1.4.2}
\end{align}
or $n_1=n_5,\ n_2=n_6,\ n_3=n_7,\ n_4=n_8$
such that
\begin{align}
&\bm{a}^T\bm{\nu'}\bm{d}=\bm{\lambda}, \quad &\bm{b}^T\bm{\gamma'}\bm{c}&=\bm{\mu},\label{Aut1.5.1}\\
&\bm{e}^T\bm{\lambda'}\bm{h}=\bm{\nu}, \quad &\bm{f}^T\bm{\mu'}\bm{g}&=\bm{\gamma},\label{Aut1.5.2}
\end{align}
or $n_1=n_6,\ n_2=n_5,\ n_3=n_8,\ n_4=n_7$
such that
\begin{align}
&\bm{a}^T\bm{\gamma'}\bm{d}=\bm{\lambda}, \quad &\bm{b}^T\bm{\nu'}\bm{c}&=\bm{\mu},\label{Aut1.6.1}\\
&\bm{e}^T\bm{\mu'}\bm{h}=\bm{\nu}, \quad &\bm{f}^T\bm{\lambda'}\bm{g}&=\bm{\gamma},\label{Aut1.6.2}
\end{align}
or $n_1=n_7,\ n_2=n_8,\ n_3=n_5,\ n_4=n_6$
such that
\begin{align}
&\bm{a}^T(\bm{\gamma'})^T\bm{d}=\bm{\lambda}, \quad &\bm{b}^T(\bm{\nu'})^T\bm{c}&=\bm{\mu},\label{Aut1.7.1}\\
&\bm{e}^T(\bm{\mu'})^T\bm{h}=\bm{\nu}, \quad &\bm{f}^T(\bm{\lambda'})^T\bm{g}&=\bm{\gamma},\label{Aut1.7.2}
\end{align}
or $n_1=n_8,\ n_2=n_7,\ n_3=n_6,\ n_4=n_5$
such that
\begin{align}
&\bm{a}^T(\bm{\nu'})^T\bm{d}=\bm{\lambda}, \quad &\bm{b}^T(\bm{\gamma'})^T\bm{c}&=\bm{\mu},\label{Aut1.8.1}\\
&\bm{e}^T(\bm{\lambda'})^T\bm{h}=\bm{\nu}, \quad &\bm{f}^T(\bm{\mu'})^T\bm{g}&=\bm{\gamma},\label{Aut1.8.2}
\end{align}
or $n_1=n_3,\ n_2=n_4$
satisfying $(\ref{Aut1.3.1})$
or $n_1=n_4,\ n_2=n_3$, $n_5=n_6,\ n_7=n_8$
satisfying $(\ref{Aut1.2.2})$, $(\ref{Aut1.4.1})$
or $n_5=n_7,\ n_6=n_8$
satisfying $(\ref{Aut1.3.2})$
or $n_1=n_2,\ n_3=n_4$, $n_5=n_8,\ n_6=n_7$
satisfying $(\ref{Aut1.2.1})$, $(\ref{Aut1.4.2})$
or $n_1=n_3=n_5=n_7,\ n_2=n_4=n_6=n_8$
satisfying $(\ref{Aut1.7.1})$ and
\begin{align}
&\bm{e}^T\bm{\lambda'}\bm{h}=\bm{\nu}, \quad &\bm{f}^T\bm{\mu'}\bm{g}&=\bm{\gamma},\label{Aut1.9}
\end{align}
or $n_1=n_3=n_6=n_8,~ n_2=n_4=n_5=n_7$
satisfying $(\ref{Aut1.8.1})$ and
\begin{align}
&\bm{e}^T\bm{\mu'}\bm{h}=\bm{\nu}, \quad &\bm{f}^T\bm{\lambda'}\bm{g}&=\bm{\gamma},\label{Aut1.10}
\end{align}
or $n_1=n_3=n_5=n_7, ~n_2=n_4=n_6=n_8$
satisfying $(\ref{Aut1.5.1})$ and
\begin{align}
&\bm{e}^T(\bm{\mu'})^T\bm{h}=\bm{\nu}, \quad &\bm{f}^T(\bm{\lambda'})^T\bm{g}&=\bm{\gamma},\label{Aut1.11}
\end{align}
or $n_1=n_3=n_6=n_8, ~n_2=n_4=n_5=n_7$
satisfying $(\ref{Aut1.6.1})$ and
\begin{align}
&\bm{e}^T(\bm{\lambda'})^T\bm{h}=\bm{\nu}, \quad &\bm{f}^T(\bm{\mu'})^T\bm{g}&=\bm{\gamma}.\label{Aut1.12}
\end{align}

$(3) ~\mathfrak{U}_2(n_1,n_2,n_3,n_4;I_2)\cong
\mathfrak{U}_2(n_1,n_2,n_3,n_4;I_2^{'})$ if and only if there exist
nonzero parameter $\beta$ satisfying $\beta^2 \lambda^{'}=\lambda$
and invertible matrices
\begin{align*}
&\bm{e}=(e_{m^{'}m})_{n_1\times n_1}, \ &\bm{f}&=(f_{q^{'}q})_{n_2\times n_2},
\ &\bm{g}&=(g_{s^{'}s})_{n_3\times n_3},\ &\bm{h}&=(h_{r^{'}r})_{n_4\times n_4},
\end{align*}
satisfying $(\ref{Aut1.1.2})$ or
$n_1=n_2,~n_3=n_4$
satisfying $(\ref{Aut1.2.2})$
or $n_1=n_3,~ n_2=n_4$
satisfying $(\ref{Aut1.3.2})$
or $n_1=n_4, ~n_2=n_3$
satisfying $(\ref{Aut1.4.2})$.

$(4)$ $\mathfrak{U}_9(n_1,n_2,n_3,n_4;I_9)\cong
\mathfrak{U}_9(n_1,n_2,n_3,n_4;I_9^{'})$ if and only if  there exist
nonzero parameter $z$ and invertible matrices
\begin{align*}
&\bm{e}=(e_{m^{'}m})_{n_1\times n_1}, \ &\bm{f}&=(f_{q^{'}q})_{n_2\times n_2},
\ &\bm{g}&=(g_{s^{'}s})_{n_3\times n_3},\ &\bm{h}&=(h_{r^{'}r})_{n_4\times n_4},
\end{align*}
satisfying $(\ref{Aut1.1.2})$ and
\begin{align}
&z\bm{e}^T\bm{\lambda'}=\bm{\lambda},  &z\bm{f}^T\bm{\mu'}&=\bm{\mu},  &z\bm{g}^T\bm{\alpha'}&=\bm{\alpha},  &z\bm{h}^T\bm{\beta'}&=\bm{\beta},
\end{align}
or
\begin{align}
&z\xi\bm{e}^T\bm{\lambda'}=\bm{\lambda},  &-z\xi\bm{f}^T\bm{\mu'}&=\bm{\mu},  &z\xi\bm{g}^T\bm{\alpha'}&=\bm{\alpha},  &-z\xi\bm{h}^T\bm{\beta'}&=\bm{\beta},
\end{align}
or $n_1=n_2,\ n_3=n_4$ and there exist invertible matrices
$\bm{e}=(e_{q^{'}m})_{n_1\times n_1}$, $\bm{f}=(f_{m^{'}q})_{n_1\times n_1}$,
$\bm{g}=(g_{r^{'}s})_{n_3\times n_3}$, $\bm{h}=(h_{s^{'}r})_{n_3\times n_3}$
satisfying $(\ref{Aut1.2.2})$ and
\begin{align}
&z\bm{e}^T\bm{\mu'}=\bm{\lambda},  &z\bm{f}^T\bm{\lambda'}&=\bm{\mu},  &z\bm{g}^T\bm{\beta'}&=\bm{\alpha},  &z\bm{h}^T\bm{\alpha'}&=\bm{\beta},
\end{align}
or
\begin{align}
&-z\xi\bm{e}^T\bm{\mu'}=\bm{\lambda},  &z\xi\bm{f}^T\bm{\lambda'}&=\bm{\mu},  &-z\xi\bm{g}^T\bm{\beta'}&=\bm{\alpha},  &z\xi\bm{h}^T\bm{\alpha'}&=\bm{\beta},
\end{align}
or $n_1=n_3,\ n_2=n_4$ and there exist invertible matrices
$\bm{e}=(e_{s^{'}m})_{n_1\times n_1}$, $\bm{f}=(f_{r^{'}q})_{n_2\times n_2}$,
$\bm{g}=(g_{m^{'}s})_{n_1\times n_1}$, $\bm{h}=(h_{q^{'}r})_{n_2\times n_2}$
satisfying $(\ref{Aut1.3.2})$ and
\begin{align}
&z\bm{e}^T\bm{\alpha'}=\bm{\lambda},  &z\bm{f}^T\bm{\beta'}&=\bm{\mu},  &z\bm{g}^T\bm{\lambda'}&=\bm{\alpha},  &z\bm{h}^T\bm{\mu'}&=\bm{\beta},
\end{align}
or
\begin{align}
&z\xi\bm{e}^T\bm{\alpha'}=\bm{\lambda},  &-z\xi\bm{f}^T\bm{\beta'}&=\bm{\mu},  &z\xi\bm{g}^T\bm{\lambda'}&=\bm{\alpha},  &-z\xi\bm{h}^T\bm{\mu'}&=\bm{\beta},
\end{align}
or $n_1=n_4,\ n_2=n_3$ and there exist invertible matrices
$\bm{e}=(e_{r^{'}m})_{n_1\times n_1}$, $\bm{f}=(f_{s^{'}q})_{n_2\times n_2}$,
$\bm{g}=(g_{q^{'}s})_{n_2\times n_2}$, $\bm{h}=(h_{m^{'}r})_{n_1\times n_1}$
satisfying $(\ref{Aut1.4.2})$ and
\begin{align}
&z\bm{e}^T\bm{\beta'}=\bm{\lambda},  &z\bm{f}^T\bm{\alpha'}&=\bm{\mu},  &z\bm{g}^T\bm{\mu'}&=\bm{\alpha},  &z\bm{h}^T\bm{\lambda'}&=\bm{\beta},
\end{align}
or
\begin{align}
&-z\xi\bm{e}^T\bm{\beta'}=\bm{\lambda},  &z\xi\bm{f}^T\bm{\alpha'}&=\bm{\mu},  &-z\xi\bm{g}^T\bm{\mu'}&=\bm{\alpha},  &z\xi\bm{h}^T\bm{\lambda'}&=\bm{\beta}.
\end{align}
\end{pro}
\begin{proof}
$(1)$ We prove the claim for $\Omega_1(n_1,n_2,\ldots,n_8)$. The proofs for other cases follow the same lines. By {\rm\cite [Theorem 6.1]{Z}}, we have that gr$(A)\cong \mathcal{B}(\Omega_1(n_1,n_2,\ldots,n_8)$
$)\sharp H$. We can check the relations listed in Definition \ref{def 1} except $(\ref{Ai})-(\ref{e1})$ hold in $A$ from the bosonization $\mathcal{B}(\Omega_1(n_1,n_2,\ldots,n_8))\sharp H$.
As $\triangle(A_i)=A_i\otimes 1+ab\otimes A_i, ~\triangle(D_\ell)=D_\ell\otimes 1+ac\otimes D_\ell,$ so
\begin{flalign*}
&\triangle(A_{i_1}A_{i_2}+A_{i_2}A_{i_1})=(A_{i_1}A_{i_2}+A_{i_2}A_{i_1})\otimes 1+1\otimes (A_{i_1}A_{i_2}+A_{i_2}A_{i_1}),\\
&\triangle(A_iD_\ell+D_\ell A_i)=(A_iD_\ell+D_\ell A_i)\otimes 1+bc\otimes (A_iD_\ell+D_\ell A_i).
\end{flalign*}
Then $A_iD_\ell+D_\ell A_i\in \mathcal{P}_{1,bc}(\mathcal{B}(\Omega_1(n_1,n_2,\ldots,n_8))\sharp H)=\mathcal{P}_{1,bc}(H)=\mathbbm{k}\{1-bc\}$. It follows that $A_{i_1}A_{i_2}+A_{i_2}A_{i_1}=0,~A_iD_\ell+D_\ell A_i=\lambda_{i,\ell}(1-bc)$ for some $\lambda_{i,\ell}\in\mathbbm{k}$.

Similarly, other relations hold in $A$. Then there is a surjective Hopf morphism from $\mathfrak{U}_1(n_1,n_2,\ldots,n_8;I_1)$ to $A$. We can observe that each element of $\mathfrak{U}_1(n_1,n_2,\ldots,n_8;I_1)$ can be expressed by a linear
combination of
\begin{equation*}
\begin{split}
\{A_i^{\alpha_i}B_j^{\beta_j}C_k^{\gamma_k}
D_\ell^{\eta_\ell}E_m^{\zeta_m}F_q^{\theta_q}G_s^{\lambda_s}H_r^{\mu_r}a^eb^fc^gd^h
&\mid \alpha_i,\beta_j,\gamma_k,\eta_\ell,\zeta_m,\theta_q,\lambda_s,\mu_r,\\
&\quad
e,f,g,h\in\mathbb{I}_{0,1}\}.
\end{split}
\end{equation*}
In fact, the set is a basis of
$\mathfrak{U}_1(n_1,n_2,\ldots,n_8;I_1)$ according to the Diamond Lemma \cite{G}.
Then $\dim A$ = $\dim \mathfrak{U}_1(n_1,n_2,\ldots,n_8;I_1)$ and
whence $A\cong \mathfrak{U}_1(n_1,n_2,\ldots,n_8;I_1)$.

$(2)$ Suppose that $\Phi : \mathfrak{U}_1(n_1,n_2,\ldots,n_8;I_1)\rightarrow \mathfrak{U}_1(n_1,n_2,\ldots,n_8;I_1^{'})$ is an isomorphism as
Hopf algebras.
When $\Phi|_H=\tau_1$ or $\tau_3$, $\Phi(A_i)$ is $(1,ab)$-skew primitive,
\begin{equation*}
\Phi(A_i)\in \bigoplus_{i^{'}=1}^{n_1}\mathbbm{k}A_{i^{'}}^{'} \bigoplus_{j=1}^{n_2}\mathbbm{k}B_{j}^{'}\bigoplus \mathbbm{k}(1-ab).
\end{equation*}
$aA_i = -A_ia$, $dA_i=-\xi A_id,\ dB_{j}=\xi B_{j}d$ imply that $\Phi(A_i)$ does not contain terms from $1-ab$ nor $\bigoplus_{j=1}^{n_2}\mathbbm{k}B_{j}^{'}$. So there exists an invertible matrix $(a_{i^{'}i})_{n_1\times n_1}$ such that $\Phi(A_i)=\sum\limits_{i^{'}=1}^{n_1}a_{i^{'}i}A_{i^{'}}^{'}$. Similarly
$\Phi(B_j)=\sum\limits_{j^{'}=1}^{n_2}b_{j^{'}j}B_{j^{'}}^{'}$,
\begin{flalign*}
&\Phi(C_k)=\sum\limits_{k^{'}=1}^{n_3}c_{k^{'}k}C_{k^{'}}^{'},
\quad\Phi(D_\ell)=\sum\limits_{\ell^{'}=1}^{n_4}d_{\ell^{'}\ell}D_{\ell^{'}}^{'},
\quad\Phi(E_m)=\sum\limits_{m^{'}=1}^{n_5}e_{m^{'}m}E_{m^{'}}^{'},\\
&\Phi(F_q)=\sum\limits_{q^{'}=1}^{n_6}f_{q^{'}q}F_{q^{'}}^{'},
\quad\Phi(G_s)=\sum\limits_{s^{'}=1}^{n_7}g_{s^{'}s}G_{s^{'}}^{'},
\quad\Phi(H_r)=\sum\limits_{r^{'}=1}^{n_8}h_{r^{'}r}H_{r^{'}}^{'}.
\end{flalign*}
Then we get that
\begin{flalign*}
&\sum\limits_{i^{'}=1}^{n_1}\sum\limits_{\ell^{'}=1}^{n_4}
a_{i^{'}i}d_{\ell^{'}\ell}\lambda_{i^{'},\ell^{'}}^{'}=\lambda_{i,\ell},
\hspace{5em}\sum\limits_{j^{'}=1}^{n_2}\sum\limits_{k^{'}=1}^{n_3}
b_{j^{'}j}c_{k^{'}k}\mu_{j^{'},k^{'}}^{'}=\mu_{j,k},\\
&\sum\limits_{m^{'}=1}^{n_5}\sum\limits_{r^{'}=1}^{n_8}
e_{m^{'}m}h_{r^{'}r}\nu_{m^{'},r^{'}}^{'}=\nu_{m,r},
\hspace{2.8em}\sum\limits_{q^{'}=1}^{n_6}\sum\limits_{s^{'}=1}^{n_7}
f_{q^{'}q}g_{s^{'}s}\gamma_{q^{'},s^{'}}^{'}=\gamma_{q,s}.
\end{flalign*}
Hence the relations $(\ref{Aut1.1.1})$, $(\ref{Aut1.1.2})$ hold.
Similarly, $\Phi$ is an isomorphism of Hopf algebras if and only if

\hspace{1em}when $\Phi|_H=\tau_2$ or $\tau_4$, the relations $(\ref{Aut1.2.1})$, $(\ref{Aut1.2.2})$ hold;

\hspace{1em}when $\Phi|_H=\tau_5$ or $\tau_7$, the relations $(\ref{Aut1.3.1})$, $(\ref{Aut1.3.2})$ hold;

\hspace{1em}when $\Phi|_H=\tau_6$ or $\tau_8$, the relations $(\ref{Aut1.4.1})$, $(\ref{Aut1.4.2})$ hold;

\hspace{1em}when $\Phi|_H=\tau_9$ or $\tau_{11}$, the relations $(\ref{Aut1.5.1})$, $(\ref{Aut1.5.2})$ hold;

\hspace{1em}when $\Phi|_H=\tau_{10}$ or $\tau_{12}$, the relations $(\ref{Aut1.6.1})$, $(\ref{Aut1.6.2})$ hold;

\hspace{1em}when $\Phi|_H=\tau_{13}$ or $\tau_{15}$, the relations $(\ref{Aut1.7.1})$, $(\ref{Aut1.7.2})$ hold;

\hspace{1em}when $\Phi|_H=\tau_{14}$ or $\tau_{16}$, the relations $(\ref{Aut1.8.1})$, $(\ref{Aut1.8.2})$ hold;

\hspace{1em}when $\Phi|_H=\tau_{17}$ or $\tau_{18}$, the relation $(\ref{Aut1.3.1})$ holds;

\hspace{1em}when $\Phi|_H=\tau_{19}$ or $\tau_{20}$, the relations $(\ref{Aut1.2.2})$, $(\ref{Aut1.4.1})$ hold;

\hspace{1em}when $\Phi|_H=\tau_{21}$ or $\tau_{22}$, the relation $(\ref{Aut1.3.2})$ holds;

\hspace{1em}when $\Phi|_H=\tau_{23}$ or $\tau_{24}$, the relations $(\ref{Aut1.2.1})$, $(\ref{Aut1.4.2})$ hold;

\hspace{1em}when $\Phi|_H=\tau_{25}$ or $\tau_{26}$, the relations $(\ref{Aut1.7.1})$, $(\ref{Aut1.9})$ hold;

\hspace{1em}when $\Phi|_H=\tau_{27}$ or $\tau_{28}$, the relations $(\ref{Aut1.8.1})$, $(\ref{Aut1.10})$ hold;

\hspace{1em}when $\Phi|_H=\tau_{29}$ or $\tau_{30}$, the relations $(\ref{Aut1.5.1})$, $(\ref{Aut1.11})$ hold;

\hspace{1em}when $\Phi|_H=\tau_{31}$ or $\tau_{32}$, the relations $(\ref{Aut1.6.1})$, $(\ref{Aut1.12})$ hold.

$(3)$ Suppose that $\Phi: \mathfrak{U}_2(n_1,n_2,n_3,n_4;I_2)\rightarrow \mathfrak{U}_2(n_1,n_2,n_3,n_4;I_2^{'})$ is an isomorphism of
Hopf algebras. Then $\Phi|_H\in \{\tau_1,\ldots,\tau_{8},\tau_{17},\ldots,\tau_{24}\}$. If $\Phi|_H\in\{\tau_1,\tau_3,\tau_{17},\tau_{18}\}$, then $\Phi(p_1)=\beta p_1^{'}$,
$\Phi(p_2)=\beta p_2^{'}$,
whence the relations $\beta^2 \lambda^{'}=\lambda$ and $(\ref{Aut1.1.2})$ hold.

\hspace{1em}When $\Phi|_H\in\{\tau_2,\tau_4,\tau_{19},\tau_{20}\}$, the relations $\beta^2 \lambda^{'}=\lambda$ and $(\ref{Aut1.2.2})$ hold;

\hspace{1em}when $\Phi|_H\in\{\tau_5,\tau_{7},\tau_{21},\tau_{22}\}$, the relations $\beta^2 \lambda^{'}=\lambda$ and $(\ref{Aut1.3.2})$ hold;

\hspace{1em}when $\Phi|_H\in\{\tau_{6},\tau_{8},\tau_{23},\tau_{24}\}$, the relations $\beta^2 \lambda^{'}=\lambda$ and $(\ref{Aut1.4.2})$ hold.

The proof for $(4)$ is completely analogous.
\end{proof}

\begin{defi}\label{def 14}
For a set of parameters $I_{14}=\{\lambda,\ \mu,\ \alpha \}$, denote by $\mathfrak{U}_{14}(I_{14})$ the algebra that is generated by $a,\ b,\ c,\ d$, $p_1,\ p_2$, $q_1,\ q_2$, satisfying the relations $(\ref{3.1})-(\ref{3.2})$, $(\ref{e2.4})-(\ref{e2.1})$ and
\begin{flalign}
& & aq_1=q_1a,\ \ bq_1=-q_1b,\ \ cq_1=-q_1c,\ \ dq_1=q_2d,\label{e14.1}\\
& & aq_2=q_2a,\ \ bq_2=-q_2b,\ \ cq_2=-q_2c,\ \ dq_2=q_1d,\label{e14.2}\\
& & q_1^2=0,\ \ q_2^2=0,\ \ q_1q_2+q_2q_1=\mu(1-bc),\label{e14.3}\\
& & p_1q_1+q_1p_1=0,\ \ p_2q_2+q_2p_2=0,\quad p_1q_2+q_2p_1=p_2q_1+q_1p_2=\alpha(1-bc).\label{e14.4}
\end{flalign}
It is a Hopf algebra with its coalgebra structure determined by
\begin{align*}
&\triangle(p_1)=p_1\otimes 1+b\otimes p_1,
\hspace{2.8em}\triangle(p_2)=p_2\otimes 1+c\otimes p_2,\\
&\triangle(q_1)=q_1\otimes 1+b\otimes q_1,
\hspace{2.8em}\triangle(q_2)=q_2\otimes 1+c\otimes q_2.
\end{align*}
\end{defi}

\begin{defi}\label{def 15}
For a set of parameters $I_{15}=\{\lambda,\ \mu,\ \alpha,\ \nu \}$, denote by $\mathfrak{U}_{15}(I_{15})$ the algebra that is generated by $a,\ b,\ c,\ d$, $p_1,\ p_2$, $q_1,\ q_2$, satisfying the relations $(\ref{3.1})-(\ref{3.2})$, $(\ref{e2.4})-(\ref{e2.1})$, $(\ref{e14.1})-(\ref{e14.3})$ and
\begin{flalign*}
& & p_1q_1+q_1p_1=p_2q_2+q_2p_2=\nu(1-a),\ p_1q_2+q_2p_1=p_2q_1+q_1p_2=\alpha(1-abc).
\end{flalign*}
It is a Hopf algebra with its coalgebra structure determined by
\begin{align*}
&\triangle(p_1)=p_1\otimes 1+b\otimes p_1,
\hspace{2.8em}\triangle(p_2)=p_2\otimes 1+c\otimes p_2,\\
&\triangle(q_1)=q_1\otimes 1+ab\otimes q_1,
\hspace{2.8em}\triangle(q_2)=q_2\otimes 1+ac\otimes q_2.
\end{align*}
\end{defi}

\begin{defi}\label{def 16}
Denote by $\mathfrak{U}_{16}(\lambda)$ the algebra that is generated by $a,\ b,\ c,\ d$, $p_1,\ p_2$, $q_1,\ q_2$, satisfying the relations $(\ref{3.1})-(\ref{3.2})$, $(\ref{e2.4})-(\ref{e2.1})$ and
\begin{flalign}
& & aq_1=-q_1a,\hspace{1em}bq_1=q_1b,\hspace{1em}cq_1=q_1c,\hspace{1em}dq_1=q_2d,\label{e16.1}\\
& & aq_2=-q_2a,\hspace{1em}bq_2=q_2b,\hspace{1em}cq_2=q_2c,\hspace{1em}dq_2=-q_1d,\label{e16.2}\\
& & q_1^2=0,\hspace{2em}q_2^2=0,\hspace{2em}q_1q_2+q_2q_1=0,\nonumber\\
& & p_1q_1-q_1p_1=0,\ p_2q_2-q_2p_2=0,\ p_1q_2-q_2p_1=0, \ p_2q_1-q_1p_2=0.\nonumber
\end{flalign}
It is a Hopf algebra with its coalgebra structure determined by
\begin{align*}
&\triangle(p_1)=p_1\otimes 1+b\otimes p_1,
\hspace{2.8em}\triangle(p_2)=p_2\otimes 1+c\otimes p_2,\\
&\triangle(q_1)=q_1\otimes 1+a\otimes q_1,
\hspace{2.8em}\triangle(q_2)=q_2\otimes 1+abc\otimes q_2.
\end{align*}
\end{defi}

\begin{defi}\label{def 20}
Denote by $\mathfrak{U}_{20}(\lambda)$ the algebra that is generated by $a,\ b,\ c,\ d$, $p_1,\ p_2$, $q_1,\ q_2$, satisfying the relations $(\ref{3.1})-(\ref{3.2})$ and
\begin{flalign}
& & ap_1=p_1a,\hspace{1em}bp_1=p_1b,\hspace{1em}cp_1=-p_1c,\hspace{1em}dp_1=p_2ad,\label{e20.1}\\
& & ap_2=p_2a,\hspace{1em}bp_2=-p_2b,\hspace{1em}cp_2=p_2c,\hspace{1em}dp_2=p_1ad,\label{e20.2}\\
& & aq_1=q_1a,\hspace{1em}bq_1=-q_1b,\hspace{1em}cq_1=q_1c,\hspace{1em}dq_1=q_2ad,\label{e20.3}\\
& & aq_2=q_2a,\hspace{1em}bq_2=q_2b,\hspace{1em}cq_2=-q_2c,\hspace{1em}dq_2=q_1ad,\label{e20.4}\\
& & p_1^2=0,\hspace{1em}p_2^2=0,\hspace{1em}p_1p_2+p_2p_1=0,~
\hspace{1em}q_1^2=0,\hspace{1em}q_2^2=0,\hspace{1em}q_1q_2+q_2q_1=0,\label{e20.5}\\
& & p_1q_1+q_1p_1=0,\ p_2q_2+q_2p_2=0,\ p_1q_2+q_2p_1=p_2q_1+q_1p_2=\lambda(1-a).\label{e20.6}
\end{flalign}
It is a Hopf algebra with its coalgebra structure determined by
\begin{align*}
&\triangle(p_1)=p_1\otimes 1+bc\otimes p_1,
\triangle(p_2)=p_2\otimes 1+abc\otimes p_2,\\
&\triangle(q_1)=q_1\otimes 1+bc\otimes q_1,
~~~\triangle(q_2)=q_2\otimes 1+abc\otimes q_2.
\end{align*}
\end{defi}

\begin{defi}\label{def 23}
Denote by $\mathfrak{U}_{23}(\lambda)$ the algebra that is generated by $a,\ b,\ c,\ d$, $p_1,\ p_2$, $q_1,\ q_2$, satisfying the relations $(\ref{3.1})-(\ref{3.2})$, $(\ref{e20.1})-(\ref{e20.4})$ and
\begin{flalign}
& & p_1^2=0,\hspace{1em}p_2^2=0,\hspace{1em}p_1p_2-p_2p_1=0,~
\hspace{1em}q_1^2=0,\hspace{1em}q_2^2=0,\hspace{1em}q_1q_2-q_2q_1=0,\label{e23.1}\\
& & p_1q_1-q_1p_1=0,\ p_2q_2-q_2p_2=0,\ p_1q_2+q_2p_1=p_2q_1+q_1p_2=\lambda(1-a).\nonumber
\end{flalign}
It is a Hopf algebra with its coalgebra structure determined by
\begin{align*}
&\triangle(p_1)=p_1\otimes 1+c\otimes p_1,
\ \ \triangle(p_2)=p_2\otimes 1+ab\otimes p_2,\\
&\triangle(q_1)=q_1\otimes 1+b\otimes q_1,
\ \ \triangle(q_2)=q_2\otimes 1+ac\otimes q_2.
\end{align*}
\end{defi}

\begin{defi}\label{def 29}
Denote by $\mathfrak{U}_{29}(\lambda)$ the algebra that is generated by $a,\ b,\ c,\ d$, $p_1,\ p_2$, $q_1,\ q_2$, satisfying the relations $(\ref{3.1})-(\ref{3.2})$, $(\ref{e16.1})$, $(\ref{e16.2})$, $(\ref{e20.5})$ and
\begin{flalign*}
& & ap_1=-p_1a,\ bp_1=p_1b,\ cp_1=p_1c,\ dp_1=p_2d,\\
& & ap_2=-p_2a,\ bp_2=p_2b,\ cp_2=p_2c,\ dp_2=-p_1d,\\
& & p_1q_1+q_1p_1=0,\quad p_2q_2+q_2p_2=0,\\
& & p_1q_2+q_2p_1=\lambda(1-bc),\quad  p_2q_1+q_1p_2=-\lambda(1-bc).
\end{flalign*}
It is a Hopf algebra with its coalgebra structure determined by
\begin{align*}
&\triangle(p_1)=p_1\otimes 1+a\otimes p_1,
\quad \triangle(p_2)=p_2\otimes 1+abc\otimes p_2,\\
&\triangle(q_1)=q_1\otimes 1+a\otimes q_1,
\quad \triangle(q_2)=q_2\otimes 1+abc\otimes q_2.
\end{align*}
\end{defi}

\begin{pro}\label{14}
$(1)$ Suppose $A$ is a finite-dimensional Hopf algebra with the coradical  $H$ such
that its infinitesimal braiding is isomorphic to $\Omega_{i}$ for $i\in\{14,15,16,20,23,29\}$, then $A\cong \mathfrak{U}_{14}(I_{14})$ or
$\mathfrak{U}_{15}(I_{15})$ or $\mathfrak{U}_{j}(\lambda)$ for $j\in\{16,20,23,29\}$.

$(2)\ \mathfrak{U}_{14}(I_{14})\cong \mathfrak{U}_{14}(I_{14}^{'})$
if and only if  there exist nonzero parameters $z_1,\ z_2$,
$\beta_1,\ \beta_2$ such that
\begin{align}
\begin{split}
&z_1^2\lambda^{'}+2z_1z_2\alpha^{'}+z_2^2\mu^{'}=\lambda,
\hspace{2em}\beta_1^2\lambda^{'}+2\beta_1\beta_2\alpha^{'}+\beta_2^2\mu^{'}=\mu,\\
&z_1\beta_1\lambda^{'}+z_2\beta_2\mu^{'}+(z_1\beta_2+z_2\beta_1)\alpha^{'}=\alpha.
\label{Aut14.1}
\end{split}
\end{align}

$(3) ~\mathfrak{U}_{15}(I_{15})\cong \mathfrak{U}_{15}(I_{15}^{'})$
if and only if  there exist nonzero parameters $z$, $\beta$ such
that
\begin{align}
\begin{split}
&z^2\lambda^{'}=\lambda,\hspace{2em}\beta^2\mu^{'}=\mu,
\hspace{2em}z\beta\nu^{'}=\nu,\hspace{2em}z\beta\alpha^{'}=\alpha,\label{Aut15.1}
\end{split}
\end{align}
or
\begin{align}
\begin{split}
&z^2\mu^{'}=\lambda,\hspace{2em}\beta^2\lambda^{'}=\mu,
\hspace{2em}z\beta\nu^{'}=\nu,\hspace{2em}z\beta\alpha^{'}=\alpha,\label{Aut15.2}
\end{split}
\end{align}
or
\begin{align}
\begin{split}
&z^2\lambda^{'}=\lambda,\hspace{2em}\beta^2\mu^{'}=\mu,
\hspace{2em}z\beta\alpha^{'}=\nu,\hspace{2em}z\beta\nu^{'}=\alpha,\label{Aut15.3}
\end{split}
\end{align}
or
\begin{align}
\begin{split}
&z^2\mu^{'}=\lambda,\hspace{2em}\beta^2\lambda^{'}=\mu,
\hspace{2em}z\beta\alpha^{'}=\nu,\hspace{2em}z\beta\nu^{'}=\alpha.\label{Aut15.4}
\end{split}
\end{align}

\end{pro}
\begin{proof} $(1)$ We prove the claim for $\Omega_{14}$. The proofs for $\Omega_{15}$, $\Omega_{16}$, $\Omega_{20}$, $\Omega_{23}$, $\Omega_{29}$ follow the same lines. By {\rm\cite [Theorem 6.1]{Z}}, we have that gr$(A)\cong \mathcal{B}(\Omega_{14})\sharp H$. Let $M_1=\mathbbm{k}\{p_1,p_2\}=\mathbbm{k}\{q_1,q_2\}$. By Proposition \ref{1}, we only need to prove that $(\ref{e14.4})$ holds in $A$. After a direct computation, we have that
\begin{flalign*}
&\triangle(p_1q_1+q_1p_1)=(p_1q_1+q_1p_1)\otimes 1+1\otimes(p_1q_1+q_1p_1),\\
&\triangle(p_2q_2+q_2p_2)=(p_2q_2+q_2p_2)\otimes 1+1\otimes(p_2q_2+q_2p_2),\\
&\triangle(p_1q_2+q_2p_1)=(p_1q_2+q_2p_1)\otimes1
+bc\otimes(p_1q_2+q_2p_1),\\
&\triangle(p_2q_1+q_1p_2)=(p_2q_1+q_1p_2)\otimes1
+bc\otimes(p_2q_1+q_1p_2).
\end{flalign*}
Since $d(p_1q_2+q_2p_1)=(p_2q_1+q_1p_2)d$, the relation $(\ref{e14.4})$ holds in $A$. Then there is a surjective Hopf morphism from $\mathfrak{U}_{14}(I_{14})$ to $A$. We can observe that all elements of $\mathfrak{U}_{14}(I_{14})$ can be expressed by linear
combinations of
\begin{equation*}
\{p_1^ip_2^jq_1^kq_2^l a^eb^fc^gd^h\mid
i,\ j,\ k,\ l,\ e,\ f,\ g,\ h\in\mathbb{I}_{0,1}\}.
\end{equation*}
According to the Diamond Lemma, the set is a basis of
$\mathfrak{U}_{14}(I_{14})$. Then $\dim A$ = $\dim \mathfrak{U}_{14}(I_{14})$,
whence $A\cong \mathfrak{U}_{14}(I_{14})$.

$(2)$ Suppose $\Phi : \mathfrak{U}_{14}(I_{14})\rightarrow \mathfrak{U}_{14}(I_{14}^{'})$ is an isomorphism of
Hopf algebras. Similar to the proof of Proposition \ref{1}, $\Phi|_H\in \{\tau_1,\ldots,\tau_{8},\tau_{17},\ldots,\tau_{24}\}$
and there exist nonzero parameters $z_1,~z_2$, $\beta_1,~\beta_2$
such that
when $\Phi|_H\in\{\tau_1,\ldots,\tau_{4},\tau_{17},\ldots,\tau_{20}\}$,
\begin{equation*}
\Phi(p_1)=z_1p_1^{'}+z_2q_1^{'},\ \Phi(p_2)=z_1p_2^{'}+z_2q_2^{'}, \ \Phi(q_1)=\beta_1p_1^{'}+\beta_2q_1^{'}, \ \Phi(q_2)=\beta_1p_2^{'}+\beta_2q_2^{'};
\end{equation*}
when $\Phi|_H\in\{\tau_5,\ldots,\tau_{8},\tau_{21},\ldots,\tau_{24}\}$,
\begin{equation*}
\Phi(p_1)=z_1p_2^{'}+z_2q_2^{'},\ \Phi(p_2)=z_1p_1^{'}+z_2q_1^{'}, \ \Phi(q_1)=\beta_1p_2^{'}+\beta_2q_2^{'}, \ \Phi(q_2)=\beta_1p_1^{'}+\beta_2q_1^{'},
\end{equation*}
then the relation $(\ref{Aut14.1})$ holds.

The proof of $(3)$ is completely analogous.
\end{proof}

\begin{pro}\label{19}
Suppose $A$ is a finite-dimensional Hopf algebra with the coradical $H$ such
that its infinitesimal braiding $V$ is isomorphic to $\Omega_i$ for $i\in\{4,5,19,21,22,30\}$, then $A\cong \mathcal{B}(V)\sharp H$.

\end{pro}
\begin{proof} We prove the claim for $\Omega_4(n_1,n_2,n_3,n_4)$. The proofs for others follow the same lines. Let $M_3=\mathbbm{k}\{p_1,p_2\}$. By {\rm\cite [Theorem 6.1]{Z}}, we have that gr$(A)\cong \mathcal{B}(\Omega_4(n_1,n_2,n_3,n_4))\sharp H$. Recall that $\mathcal{B}(\Omega_4(n_1,n_2,n_3,n_4))\sharp H$ is the algebra generated by \begin{equation*}
\{A_i\},\ \{B_j\},\ \{E_m\},\ \{F_q\},\ p_1,\ p_2,\ a,\ b,\ c,\ d
\end{equation*}
with
$\{A_i\}$, $\{B_j\}$, $\{E_m\}$, $\{F_q\}$, $p_1,\ p_2$, satisfying the relations of $\mathcal{B}(\Omega_4(n_1,n_2,n_3,$
$n_4))$, $a,\ b,\ c,\ d$, satisfying the relations of $H$, and all together, satisfying the relations that give the commutativity: (\ref{e20.1}), (\ref{e20.2}).
%
By Propositions $\ref{1}$ and $\ref{14}$, we know that the relations
\begin{flalign*}
&A_iB_j+B_jA_i=0, \ A_iE_m-E_mA_i=0,\  A_iF_q-F_qA_i=0, \ B_jE_m-E_mB_j=0,\\
&B_jF_q-F_qB_j=0,\  E_mF_q+F_qE_m=0, \ p_1^2=0, ~~p_2^2=0, \ p_1p_2+p_2p_1=0
\end{flalign*}
hold in $A$.
As
\begin{equation*}
\triangle(p_1)=p_1\otimes 1+bc\otimes p_1,
\ \triangle(p_2)=p_2\otimes 1+abc\otimes p_2,
\ \triangle(A_i)=A_i\otimes 1+ab\otimes A_i,
\end{equation*}
we have that
\begin{equation*}
\triangle(p_1A_i-A_ip_1)=(p_1A_i-A_ip_1)\otimes 1+ac\otimes(p_1A_i-A_ip_1).
\end{equation*}
Since $a(p_1A_i-A_ip_1)=-(p_1A_i-A_ip_1)a$,
then $p_1A_i-A_ip_1=0$ hold in $A$.
Similarly,
\begin{flalign*}
&p_1B_j-B_jp_1=0,\ p_1E_m-E_mp_1=0, \ p_1F_q-F_qp_1=0,\ p_2A_i+A_ip_2=0,\\ &p_2B_j+B_jp_2=0,\ p_2E_m+E_mp_2=0, \ p_2F_q+F_qp_2=0
\end{flalign*}
hold in $A$.
Therefore, $A\cong gr(A)$.
\end{proof}

\begin{defi}\label{def 38}
For a set of parameters $I_{38}=\{\lambda,~\mu,~\alpha \}$, denote by $\mathfrak{U}_{38}(I_{38})$ the algebra that is generated by $a,\ b,\ c,\ d$, $p_1,\ p_2$, $q_1,\ q_2$, satisfying the relations $(\ref{3.1})-(\ref{3.2})$ and
\begin{flalign}
& & ap_1=p_1a,\hspace{1em}bp_1=-p_1b,\hspace{1em}cp_1=-p_1c,\hspace{1em}dp_1=p_1d,\label{e38.1}\\
& & ap_2=p_2a,\hspace{1em}bp_2=p_2b,\hspace{1em}cp_2=p_2c,\hspace{1em}dp_2=-p_2d,\label{e38.2}\\
& & aq_1=q_1a,\hspace{1em}bq_1=-q_1b,\hspace{1em}cq_1=-q_1c,\hspace{1em}dq_1=q_1d,\label{e38.3}\\
& & aq_2=q_2a,\hspace{1em}bq_2=q_2b,\hspace{1em}cq_2=q_2c,\hspace{1em}dq_2=-q_2d,\label{e38.4}\\
& & p_1^2=\lambda(abc+a-2),\hspace{1em}p_2^2=\lambda(abc-a),\hspace{1em}p_1p_2+p_2p_1=0,\label{e38.5}\\
& & q_1^2=\mu(abc+a-2),\hspace{1em}q_2^2=\mu(abc-a),\hspace{1em}q_1q_2+q_2q_1=0,\label{e38.6}\\
& & p_1q_1+q_1p_1=\alpha(abc+a-2),\hspace{1em}p_2q_2+q_2p_2=\alpha(abc-a),\label{e38.7}\\
& & p_1q_2+q_2p_1=0,\hspace{1em}p_2q_1+q_1p_2=0.&\label{e38.8}
\end{flalign}
It is a Hopf algebra with its coalgebra structure determined by
\begin{align*}
&\triangle(p_1)=p_1\otimes 1+\frac{1}{2}a(b+c)d\otimes p_1+\frac{1}{2}a(b-c)d\otimes p_2,\\
&\triangle(p_2)=p_2\otimes 1+\frac{1}{2}(b+c)d\otimes p_2+\frac{1}{2}(b-c)d\otimes p_1,\\
&\triangle(q_1)=q_1\otimes 1+\frac{1}{2}a(b+c)d\otimes q_1+\frac{1}{2}a(b-c)d\otimes q_2,\\
&\triangle(q_2)=q_2\otimes 1+\frac{1}{2}(b+c)d\otimes q_2+\frac{1}{2}(b-c)d\otimes q_1.
\end{align*}
\end{defi}

\begin{defi}\label{def 39}
For a set of parameters $I_{39}=\{\lambda,\mu,\alpha \}$, denote by $\mathfrak{U}_{39}(I_{39})$ the algebra that is generated by $a,\ b,\ c,\ d$, $p_1,\ p_2$, $q_1,\ q_2$, satisfying the relations $(\ref{3.1})-(\ref{3.2})$, $(\ref{e38.1})-(\ref{e38.6})$ and
\begin{flalign}
& & p_1q_1+q_1p_1=p_2q_2+q_2p_2=\alpha(1-bc),
\ \ p_1q_2+q_2p_1=0,\ \ p_2q_1+q_1p_2=0.
\end{flalign}
It is a Hopf algebra with its coalgebra structure determined by
\begin{align*}
&\triangle(p_1)=p_1\otimes 1+\frac{1}{2}a(b+c)d\otimes p_1+\frac{1}{2}a(b-c)d\otimes p_2,\\
&\triangle(p_2)=p_2\otimes 1+\frac{1}{2}(b+c)d\otimes p_2+\frac{1}{2}(b-c)d\otimes p_1,\\
&\triangle(q_1)=q_1\otimes 1+\frac{1}{2}(b+c)d\otimes q_1+\frac{1}{2}(b-c)d\otimes q_2,\\
&\triangle(q_2)=q_2\otimes 1+\frac{1}{2}a(b+c)d\otimes q_2+\frac{1}{2}a(b-c)d\otimes q_1.
\end{align*}
\end{defi}

\begin{defi}\label{def 41}
For a set of parameters $I_{41}=\{\lambda,~\mu,~\alpha \}$, denote by $\mathfrak{U}_{41}(I_{41})$ the algebra that is generated by $a,\ b,\ c,\ d$, $p_1,\ p_2$, $q_1,\ q_2$, satisfying the relations $(\ref{3.1})-(\ref{3.2})$ and
\begin{flalign}
& & ap_1=p_1a,\hspace{1em}bp_1=p_1b,\hspace{1em}cp_1=p_1c,\hspace{1em}dp_1=-p_1d,\label{e41.1}\\
& & ap_2=p_2a,\hspace{1em}bp_2=-p_2b,\hspace{1em}cp_2=-p_2c,\hspace{1em}dp_2=-p_2d,\label{e41.2}\\
& & aq_1=q_1a,\hspace{1em}bq_1=q_1b,\hspace{1em}cq_1=q_1c,\hspace{1em}dq_1=-q_1d,\label{e41.3}\\
& & aq_2=q_2a,\hspace{1em}bq_2=-q_2b,\hspace{1em}cq_2=-q_2c,\hspace{1em}dq_2=-q_2d,\label{e41.4}\\
& & p_1^2=\lambda(abc+a-2),\hspace{1em}p_2^2=\lambda(a-abc),\hspace{1em}p_1p_2+p_2p_1=0,\label{e41.5}\\
& & q_1^2=\mu(abc+a-2),\hspace{1em}q_2^2=\mu(a-abc),\hspace{1em}q_1q_2+q_2q_1=0,\label{e41.6}\\
& & p_1q_1+q_1p_1=\alpha(abc+a-2),\hspace{1em}p_2q_2+q_2p_2=\alpha(a-abc),&\\
& & p_1q_2+q_2p_1=0,\hspace{1em}p_2q_1+q_1p_2=0.&
\end{flalign}
It is a Hopf algebra with its coalgebra structure determined by
\begin{align*}
&\triangle(p_1)=p_1\otimes 1+\frac{1}{2}a(1+bc)d\otimes p_1+\frac{1}{2}a(1-bc)d\otimes p_2,\\
&\triangle(p_2)=p_2\otimes 1+\frac{1}{2}(1+bc)d\otimes p_2+\frac{1}{2}(1-bc)d\otimes p_1,\\
&\triangle(q_1)=q_1\otimes 1+\frac{1}{2}a(1+bc)d\otimes q_1+\frac{1}{2}a(1-bc)d\otimes q_2,\\
&\triangle(q_2)=q_2\otimes 1+\frac{1}{2}(1+bc)d\otimes q_2+\frac{1}{2}(1-bc)d\otimes q_1.
\end{align*}
\end{defi}

\begin{defi}\label{def 42}
For a set of parameters $I_{42}=\{\lambda,~\mu,~\alpha \}$, denote by $\mathfrak{U}_{42}(I_{42})$ the algebra that is generated by $a,\ b,\ c,\ d$, $p_1,\ p_2$, $q_1,\ q_2$, satisfying the relations $(\ref{3.1})-(\ref{3.2})$, $(\ref{e41.1})-(\ref{e41.6})$ and
\begin{flalign*}
& & p_1q_1+q_1p_1=\alpha(1-bc),\ p_2q_2+q_2p_2=-\alpha(1-bc),
\ p_1q_2+q_2p_1=p_2q_1+q_1p_2=0.
\end{flalign*}
It is a Hopf algebra with its coalgebra structure determined by
\begin{align*}
&\triangle(p_1)=p_1\otimes 1+\frac{1}{2}a(1+bc)d\otimes p_1+\frac{1}{2}a(1-bc)d\otimes p_2,\\
&\triangle(p_2)=p_2\otimes 1+\frac{1}{2}(1+bc)d\otimes p_2+\frac{1}{2}(1-bc)d\otimes p_1,\\
&\triangle(q_1)=q_1\otimes 1+\frac{1}{2}(1+bc)d\otimes q_1+\frac{1}{2}(1-bc)d\otimes q_2,\\
&\triangle(q_2)=q_2\otimes 1+\frac{1}{2}a(1+bc)d\otimes q_2+\frac{1}{2}a(1-bc)d\otimes q_1.
\end{align*}
\end{defi}

\begin{defi}\label{def 44}
For a set of parameters $I_{44}=\{\lambda,~\mu,~\alpha \}$, denote by $\mathfrak{U}_{44}(I_{44})$ the algebra that is generated by $a,\ b,\ c,\ d$, $p_1,\ p_2$, $q_1,\ q_2$, satisfying the relations $(\ref{3.1})-(\ref{3.2})$ and
\begin{flalign}
& & ap_1=p_1a,\hspace{1em}bp_1=p_1b,\hspace{1em}cp_1=-p_1c,\hspace{1em}dp_1=-p_2ad,\label{e44.1}\\
& & ap_2=p_2a,\hspace{1em}bp_2=-p_2b,\hspace{1em}cp_2=p_2c,\hspace{1em}dp_2=-p_1ad,\label{e44.2}\\
& & aq_1=q_1a,\hspace{1em}bq_1=q_1b,\hspace{1em}cq_1=-q_1c,\hspace{1em}dq_1=-q_2ad,\label{e44.3}\\
& & aq_2=q_2a,\hspace{1em}bq_2=-q_2b,\hspace{1em}cq_2=q_2c,\hspace{1em}dq_2=-q_1ad,\label{e44.4}\\
& & p_1p_2=0,\hspace{2em}p_2p_1=0,\hspace{2em}p_1^2+p_2^2=\lambda(1-a),\label{e44.5}\\
& & q_1q_2=0,\hspace{2em}q_2q_1=0,\hspace{2em}q_1^2+q_2^2=\mu(1-a),\label{e44.6}\\
& & p_1q_1+q_1p_1+p_2q_2+q_2p_2=\alpha(1-a),\quad \quad\quad p_1q_1-q_1p_1-p_2q_2+q_2p_2=0,\label{e44.7}\\
& & p_1q_2+q_2p_1+p_2q_1+q_1p_2=0,\quad \quad\quad p_1q_2-q_2p_1-p_2q_1+q_1p_2=0.\label{e44.8}
\end{flalign}
It is a Hopf algebra with the same coalgebra structure as $\mathfrak{U}_{41}(I_{41})$.
\end{defi}

\begin{defi}\label{def 45}
For a set of parameters $I_{45}=\{\lambda,~\mu\}$, denote by $\mathfrak{U}_{45}(I_{45})$ the algebra that is generated by $a,\ b,\ c,\ d$, $p_1,\ p_2$, $q_1,\ q_2$, satisfying the relations $(\ref{3.1})-(\ref{3.2})$, $(\ref{e44.1})-(\ref{e44.6})$ and
\begin{flalign}
& & p_1q_1+q_1p_1+p_2q_2+q_2p_2=0,\hspace{1em}p_1q_1-q_1p_1-p_2q_2+q_2p_2=0,\\
& & p_1q_2+q_2p_1+p_2q_1+q_1p_2=0,\hspace{1em}p_1q_2-q_2p_1-p_2q_1+q_1p_2=0.
\end{flalign}
It is a Hopf algebra with the same coalgebra structure as $\mathfrak{U}_{42}(I_{42})$.
\end{defi}

\begin{pro}\label{38}
$(1)$ Suppose $A$ is a finite-dimensional Hopf algebra with the coradical  $H$ such
that its infinitesimal braiding is isomorphic to $\Omega_{i}$ for $i\in\{38,39,41,42,44,45\}$, then $A\cong \mathfrak{U}_{i}(I_i)$.

$(2)\ \mathfrak{U}_{38}(I_{38})\cong \mathfrak{U}_{38}(I_{38}^{'})$,
$\mathfrak{U}_{41}(I_{41})\cong \mathfrak{U}_{41}(I_{41}^{'})$,
$\mathfrak{U}_{44}(I_{44})\cong \mathfrak{U}_{44}(I_{44}^{'})$
if and only if  there exist nonzero parameters $z_1,\ z_2$,
$\beta_1,\ \beta_2$ such that
\begin{align}
\begin{split}
&z_1^2\lambda^{'}+z_1z_2\alpha^{'}+z_2^2\mu^{'}=\lambda,
\hspace{2em}\beta_1^2\lambda^{'}+\beta_1\beta_2\alpha^{'}+\beta_2^2\mu^{'}=\mu,\\
&2z_1\beta_1\lambda^{'}+2z_2\beta_2\mu^{'}+(z_1\beta_2+z_2\beta_1)\alpha^{'}=\alpha.
\label{Aut38}
\end{split}
\end{align}

$(3)\ \mathfrak{U}_{39}(I_{39})\cong \mathfrak{U}_{39}(I_{39}^{'})$,
$\mathfrak{U}_{42}(I_{42})\cong \mathfrak{U}_{42}(I_{42}^{'})$ if
and only if  there exist nonzero parameters $z$, $\beta$ such that
\begin{align}
\begin{split}
&z^2\lambda^{'}=\lambda,\hspace{2em}\beta^2\mu^{'}=\mu,
\hspace{2em}z\beta\alpha^{'}=\alpha,\label{Aut39.1}
\end{split}
\end{align}
or
\begin{align}
\begin{split}
&z^2\mu^{'}=\lambda,\hspace{2em}\beta^2\lambda^{'}=\mu,
\hspace{2em}z\beta\alpha^{'}=\alpha.\label{Aut39.2}
\end{split}
\end{align}

$(4)$  $\mathfrak{U}_{45}(\lambda,\mu)\cong \mathfrak{U}_{45}(1,1)$ for $\lambda,\ \mu\neq 0$.
\end{pro}

\begin{proof} $(1)$ We prove the claim for $\Omega_{38}$. The proofs for $\Omega_{39}$, $\Omega_{41}$, $\Omega_{42}$, $\Omega_{44}$, $\Omega_{45}$ follow the same lines. By {\rm\cite [Theorem 6.1]{Z}}, we have that gr$(A)\cong \mathcal{B}(\Omega_{38})\sharp H$. Let $M_{13}=\mathbbm{k}\{p_1,p_2\}=\mathbbm{k}\{q_1,q_2\}$.
Similar to Proposition \ref{14}, we only need to prove that  $(\ref{e38.5})-(\ref{e38.8})$ in Definition \ref{def 38} hold in $A$. After a direct computation, we have that
\begin{flalign*}
&\triangle(p_1^2+p_2^2)=(p_1^2+p_2^2)\otimes 1+abc\otimes (p_1^2+p_2^2),\\
&\triangle(p_1^2-p_2^2)=(p_1^2-p_2^2)\otimes 1+a\otimes (p_1^2-p_2^2),\\
&\triangle(p_1p_2+p_2p_1)=(p_1p_2+p_2p_1)\otimes 1+bc\otimes (p_1p_2+p_2p_1),\\
&\triangle(p_1q_1+q_1p_1)=(p_1q_1+q_1p_1)\otimes 1+\frac{1}{2}a(1+bc)\otimes (p_1q_1+q_1p_1)\\
&\hspace{8em}+\frac{1}{2}a(bc-1)\otimes (p_2q_2+q_2p_2),\\
&\triangle(p_2q_2+q_2p_2)=(p_2q_2+q_2p_2)\otimes 1+\frac{1}{2}a(1+bc)\otimes (p_2q_2+q_2p_2)\\
&\hspace{8em}+\frac{1}{2}a(bc-1)\otimes (p_1q_1+q_1p_1),\\
&\triangle(p_1q_2+q_2p_1)=(p_1q_2+q_2p_1)\otimes 1+\frac{1}{2}(1+bc)\otimes (p_1q_2+q_2p_1)\\
&\hspace{8em}+\frac{1}{2}(bc-1)\otimes (p_2q_1+q_1p_2),\\
&\triangle(p_2q_1+q_1p_2)=(p_2q_1+q_1p_2)\otimes 1+\frac{1}{2}(1+bc)\otimes (p_2q_1+q_1p_2)\\
&\hspace{8em}+\frac{1}{2}(bc-1)\otimes (p_1q_2+q_2p_1).
\end{flalign*}
As $b(p_1p_2+p_2p_1)=-(p_1p_2+p_2p_1)b$, the relations $(\ref{e38.5}),(\ref{e38.6})$ hold in $A$. Since
\begin{equation*}
b(p_1q_2+q_2p_1)=-(p_1q_2+q_2p_1)b,\hspace{2em}b(p_2q_1+q_1p_2)=-(p_2q_1+q_1p_2)b
\end{equation*}
the relations $(\ref{e38.7})$, $(\ref{e38.8})$ hold in $A$.
Then there is a surjective Hopf morphism from
$\mathfrak{U}_{38}(I_{38})$ to $A$. We can observe that all elements of $\mathfrak{U}_{38}(I_{38})$ can be expressed by linear
combinations of
\begin{equation*}
\{p_1^ip_2^jq_1^kq_2^l a^eb^fc^gd^h\mid
 i,j,k,l,e,f,g,h\in\mathbb{I}_{0,1}\}.
\end{equation*}
According to the Diamond Lemma, the set is a basis of
$\mathfrak{U}_{38}(I_{38})$. Then $\dim A$ = $\dim \mathfrak{U}_{38}(I_{38})$,
whence $A\cong \mathfrak{U}_{38}(I_{38})$.

The proofs of $(2)$, $(3)$ are similar to those of Proposition $\ref{14}\ (2),\ (3)$.

$(4)$ When $\lambda,\ \mu\neq 0$, $\Phi:\mathfrak{U}_{45}(1,1)\rightarrow \mathfrak{U}_{45}(\lambda,\mu)$, by $\Phi|_{H}={\rm id}$, $p_i\mapsto \sqrt{\lambda}p_i$,
$q_i\mapsto \sqrt{\mu}q_i$ for $i=1,2$.
\end{proof}
$\mathbf{Proof\ of\ Theorem\ B}$. Let $M$ be one of the Yetter-Drinfeld modules listed in Theorem A. Let $A$ be a finite-dimensional Hopf algebra
over $H$ such that its infinitesimal braiding is isomorphic to $M$.
By {\rm\cite [Theorem 6.1]{Z}}, gr$A\cong \mathcal{B}(M)\sharp H$.
By Propositions \ref{1},\ \ref{14},\ \ref{19},\ \ref{38}, we finish the
proof.

\begin{center}
$\mathbf{ACKNOWLEDGMENT}$
\end{center}

The authors would like to thank Dr. Rongchuan Xiong for his helpful
discussions.  The second author is supported by NSERC of Canada  the NSFC
(Grant No. 11931009). The third author is supported by the NSFC
(Grant No. 12171155) and in part by the Science and Technology
Commission of Shanghai Municipality (No. 22DZ2229014).

\end{document}